\documentclass[11pt]{article}
\usepackage [dvips]{graphics}
\usepackage[centertags]{amsmath}
\usepackage{amsfonts}
\usepackage{amssymb}
\usepackage{amsthm}
\usepackage{amsmath,amscd}
\usepackage{stmaryrd}
\usepackage[]{titletoc}
\usepackage{lipsum}
\usepackage{soul}
\usepackage{mathrsfs}
\usepackage{xypic}
\usepackage{extarrows}
\usepackage[margin=1in]{geometry}
\usepackage{multirow}
\usepackage{array}
\usepackage{tocloft}

\theoremstyle{plain}
\newtheorem{thm}{Theorem}[section]
\newtheorem{defn}[thm]{Definition}
\newtheorem{cor}[thm]{Corollary}
\newtheorem{lem}[thm]{Lemma}
\newtheorem{prop}[thm]{Proposition}
\newtheorem{exam}[thm]{Example}
\newtheorem{rem}[thm]{Remark}

\newcommand{\CC}{\mathbb{C}}
\newcommand{\bn}{\mathbb{B}_n}
\newcommand{\cn}{\mathbb{C}^n}
\newcommand{\dd}{\mathbb{D}}
\newcommand{\tori}{\mathbb{T}}
\newcommand{\bert}{L_{a,t}^{2}(\mathbb{B}_n)}

\newcommand{\sn}{\mathbb{S}_n}
\newcommand{\intd}{\mathrm{d}}
\newcommand{\la}{\langle}
\newcommand{\ra}{\rangle}
\newcommand{\bpartial}{\bar{\partial}}
\newcommand{\bz}{\bar{z}}
\newcommand{\bw}{\bar{w}}

\newcommand{\BTt}{T^{(t)}}

\newcommand{\BHt}{H^{(t)}}
\newcommand{\BKt}{K^{(t)}}

\newcommand{\BFt}{\mathcal{F}^{(t)}}
\newcommand{\BGt}{\mathcal{G}^{(t)}}
\newcommand{\BPt}{P^{(t)}}
\newcommand{\Tr}{\mathrm{Tr}}

\newcommand{\ind}{\mathbb{N}_0^n}

\newcommand{\supp}{\mathrm{supp}}

\allowdisplaybreaks[4]

\makeatletter\@addtoreset{equation}{section} \makeatother
\title {Trace Formula of Semicommutators}
\author{Xiang Tang\thanks{Department of Mathematics and Statistics, Washington University, St. Louis, MO, U.S.A., 63130, xtang@math.wustl.edu.}, Yi Wang
	\thanks{Department of Mathematics, Chongqing University, Chongqing,  China, 400044,  wang\_yi@cqu.edu.cn}, and Dechao Zheng\thanks{Department of Mathematics, Vanderbilt University, Nashville, TN, U.S.A., 37240, dechao.zheng@vanderbilt.edu.} }
\date{}

\begin{document}
	\maketitle
	
	\begin{abstract}
		For weighted Bergman spaces on the unit disk, we give trace formulas of semicommutators of Toeplitz operators with $\mathscr{C}^2(\overline{\dd})$ symbols. We generalize this formula to weighted Bergman spaces on the unit ball in higher dimensions.  Applications and examples on the Hankel operators are also discussed.
		
		~
		
	\end{abstract}
	
	
	\section{Introduction}\label{sec: introduction}
	Commutators of Toeplitz operators have been objects of interest in the study of analytic function spaces for a long time. Various properties, such as compactness, Schatten class membership, trace formulas, were studied in a numerous amount of works (c.f. \cite{Al-Pe:trace, Boya91, Ch-Si:trace, Ch-Hu:traceII, Gu-Wa-Zh:trace, Mi-Pr-Si:trace, Voicu14, Yang03, Zhubookspaces, Zhu:bookoperator}). Among others, it is well-known (c.f. \cite{HH2, Zhu2001trace}) that for relatively nice symbols $f$ and $g$ on the unit disk $\dd$, the commutator $[T_f, T_g]=T_fT_g-T_gT_f$ on the Bergman space $L_a^2(\dd)$, is in the trace class, and 
	\begin{equation}\label{eq:HH}
	\Tr[T_f, T_g]=\frac{1}{2\pi i}\int_{\dd}\intd f\wedge\intd g.
	\end{equation}
	
This elegant formula is deeply connected to the Pincus function for a pair of noncommuting selfadjoint operators, c.f. \cite{Ca-Pi:exponential, Ca-Pi:mosaics, Pi:commutators}. 

Our study of trace of semi-commutator is inspired by our investigation \cite{Ta-Wa-Zh:HHtrace} of the Connes-Chern character for the Toeplitz extension (c.f. \cite{Co:noncommutative}). Semi-commutator of Toeplitz operators is the building block in Connes construction. On the other hand, the semi-commutator has its own importance. Let $\BHt_f$ be the Hankel operator with symbol $f$. The following equation
	\begin{equation}\label{eq:hankel}
		\BTt_f\BTt_g-\BTt_{fg}=-H^{(t)*}_{\bar{f}}\BHt_g,
	\end{equation}
provides a natural link between the semi-commutators of Toeplitz operators and Hankel operators, which allows to study the Hilbert-Schmidt norm of a Hankel operator by the trace of the associated semi-commutator. We aim in article to establish a generalization of the Helton-Howe trace formula (\ref{eq:HH}) to semi-commutators, which has not been explored in literature. 
	
Suppose $f$ and $g$ are two Lipschitz functions on $\dd$. It is well-known that for any $t>-1$, the semi-commutator $T_fT_g-T_{fg}$ is in the trace class (cf. \cite{Zhubookspaces}). We will establish a trace formula for $T_fT_g-T_{fg}$ when $f,g$ are nice function. More generally, we consider all weighted Bergman spaces $L_{a,t}^2(\dd)$, $t>-1$. To distinguish from the Bergman space, we add a superscript ``$(t)$'', i.e., $\BTt_f$ to denote the Toeplitz operator on $L_{a,t}^2(\dd)$ with symbol $f$. The explicit definitions are given in Section \ref{sec: preliminary}. We obtain the following trace formula. 
	
	\begin{thm}\label{thm: semi-commutator trace in dim 1}
		Suppose $t>-1$ and $f, g\in\mathscr{C}^2(\overline{\dd})$. Then
		\begin{equation}\label{eqn: semi-commutator trace in dim 1}
			\Tr\bigg(\BTt_f\BTt_g-\BTt_{fg}\bigg)=\frac{1}{2\pi i}\int_{\dd}\partial f\wedge\bpartial g+\int_{\dd^2}\varrho_t(|\varphi_z(w)|^2)\Delta f(z)\Delta g(w)\intd m(z, w).
		\end{equation}
		Here $\varrho_t$ is defined as below and is strictly positive on $(0,1)$.
		\[
		\varrho_t(s)=\frac{t+1}{16\pi^2}\int_s^1(1-x)^tx^{-1}F(s,x)\intd x,
		\]
		where
		\[
		F(s,x)=-\bigg[x\ln\frac{s}{x}+(1-x)\ln\frac{1-s}{1-x}\bigg].
		\]
		Moreover,
		\begin{equation}\label{eqn: semi-commutator trace in dim 1 limit}
			\lim_{t\to\infty}\Tr\bigg(\BTt_f\BTt_g-\BTt_{fg}\bigg)=\frac{1}{2\pi i}\int_{\dd}\partial f\wedge\bpartial g.
		\end{equation}
	\end{thm}
	Let $\BHt_f$ be the Hankel operator with symbol $f$. By Equation (\ref{eq:hankel}),	we apply the trace formula (\ref{eqn: semi-commutator trace in dim 1})  to study the Hilbert-Schmidt norm of Hankel operators with $\mathscr{C}^2(\overline{\dd})$ symbols (See Corollary \ref{cor: hankel HS norm in dim 1} and Corollary \ref{cor: Hankel HS norm inequality in dim 1}).

	Next we generalize Theorem \ref{thm: semi-commutator trace in dim 1} to higher dimensions. In general, for $n\geq2$, semicommutators of Toeplitz operators on $\bert$ with Lipschitz symbols only belong 
	to $\mathcal{S}^p$ for $p>n$. In fact, it was shown in \cite{Zhu:HShankel} that in the case when $f=\bar{g}$ and $g$ is anti-holomorphic, their semi-commutator is in the trace class only when $g=0$. To make $\BTt_f\BTt_g-\BTt_{fg}$ belong to trace class, one generally requires some further assumptions. We do not aim to give a criterion of when semi-commutators belong to the trace class. Instead, we focus on giving a trace formula for relatively nice symbols. 
	
	Recall that the Levi form $L_zf$ of a function $f$ at a point $z$ is the two form
	\[
	L_zf(\xi)=\sum_{i,j=1}^n\partial_i\bpartial_jf(z)\xi_i\bar{\xi}_j,\quad\forall\xi\in\cn.
	\]
	Define $\partial_zf$ and $\bpartial_z f$ as the $n$-vectors that has $\partial_if(z)$  and $\bpartial_i f(z)$ in its $i$-th entry. Then
	\[
	\la\partial_zf,\overline{z-w}\ra=\sum_{i=1}^n\partial_i f(z)(z_i-w_i),\quad\la\bpartial_wf,z-w\ra=\sum_{j=1}^n\bpartial_j f(w)\overline{(z_j-w_j)}.
	\]
	We say that $f, g$ satisfy
	{\bf Condition 1} if $f, g\in\mathscr{C}^1(\overline{\bn})$ and there exist $C>0$, $\epsilon>0$, such that
	\begin{equation}
		\big|\la\partial_zf,\overline{z-w}\ra\la\bpartial_wg,z-w\ra\big|\leq C|\varphi_z(w)|^2|1-\la z, w\ra|^{n+\epsilon},\quad\forall z, w\in\bn.
	\end{equation}
	We say that $f, g$ satisfy
	{\bf Condition 2} if $f, g$ satisfy condition 1, and $f, g\in\mathscr{C}^2(\overline{\bn})$ satisfy the following inequalities. $\forall z, w\in \mathbb{B}_n$
	\begin{equation}
		\big|\la\partial_zf,\overline{z-w}\ra L_wg(z-w)\big|\leq C|\varphi_z(w)|^3|1-\la z,w\ra|^{n+\epsilon},
	\end{equation}
	\begin{equation}
		\big|L_zf(z-w)\la\bpartial_wg,z-w\ra\big|\leq C|\varphi_z(w)|^3|1-\la z,w\ra|^{n+\epsilon},
	\end{equation}
	\begin{equation}
		\big|L_zf(z-w)L_wg(z-w)\big|\leq C|\varphi_z(w)|^4|1-\la z,w\ra|^{n+\epsilon}.
	\end{equation}
Here $|\varphi_z(w)|$ is the length of the M\"{o}bius transform $\varphi_z(w)$, or the pseudo-hyperbolic distance of $z$ and $w$. We obtain the following generalization of Theorem \ref{thm: semi-commutator trace in dim 1}.
	\begin{thm}\label{thm: higher dimensions}
		If $t>2n-3$, and $f, g$ satisfy Condition 1,
		then the semicommutator $\BTt_f\BTt_g-\BTt_{fg}$ is in the trace class. If furthermore $f, g$ satisfy Condition 2,
		then
		\begin{flalign}\label{eqn: trace formula high dim}
			\Tr\bigg(\BTt_f\BTt_g-\BTt_{fg}\bigg)=&a_{n,t}\int_{\bn}\partial f\wedge\bpartial g\wedge\bigg[\partial\bpartial\log(1-|w|^2)\bigg]^{n-1}\\
			&+\int_{\bn\times\bn}\rho_{n,t}(|\varphi_z(w)|^2)L_zf(z-w)L_wg(z-w)\frac{\intd m(z,w)}{|1-\la z,w\ra|^{2n+2}}.\nonumber
		\end{flalign}
		Here
		\[
		a_{n,t}=\frac{-\int_0^1(1-s)^{n-1}s^t\ln s\intd s}{\big(B(n,t+1)^2\big)n(2\pi i)^n},
		\]
		and
		\begin{equation}\label{eqn: rhont}
			\rho_{n,t}(s)=s^{-n-1}\sum_{k=1}^n\frac{(n-1)!\Gamma^2(n+t+1)}{(n-k)!\Gamma(t+1+k)\Gamma(t+1)\pi^{2n}}\int_s^1F(s,x)x^{n-k-1}(1-x)^{t+k-1}\intd x.
		\end{equation}
		In particular, $\rho_{n,t}$ is strictly positive on $(0,1)$, and
		\begin{equation}\label{eqn: asymp trace formula high dim}
			\lim_{t\to\infty}t^{1-n}\Tr\bigg(\BTt_f\BTt_g-\BTt_{fg}\bigg)=\frac{1}{(n-1)!(2\pi i)^n}\int_{\bn}\partial f\wedge\bpartial g\wedge\bigg[\partial\bpartial\log(1-|w|^2)\bigg]^{n-1}.
		\end{equation}
	\end{thm}

\begin{rem}
One can show that if the functions $f$ and $g$ are in $\mathscr{C}^2(\overline{\dd})$ then they satisfy Condition 1 and 2 with $n=1$. Therefore Theorem \ref{thm: higher dimensions} actually implies Theorem \ref{thm: semi-commutator trace in dim 1}. Because trace formulas at dimension $1$ are of independent interest, and because the proof gets significantly more complicated at higher dimensions, we first give a complete proof of Theorem \ref{thm: semi-commutator trace in dim 1} in Section \ref{sec: trace formulas on the disk}.
\end{rem}

	\begin{rem}
		In Lemma \ref{lem: d f g= f g log = f g R} we give an alternative expression of the first term in the right hand side of \eqref{eqn: trace formula high dim}, in terms of radial derivatives $R=\sum_{i=1}^nz_i\partial_i$. So \eqref{eqn: trace formula high dim} and \eqref{eqn: asymp trace formula high dim} can be rewritten as
		\begin{flalign*}
			\Tr\bigg(\BTt_f\BTt_g-\BTt_{fg}\bigg)=&-(2i)^n(n-1)!a_{n,t}\int_{\bn}\frac{\sum_{i=1}^n\partial_if(w)\bpartial_ig(w)-Rf(w)\bar{R}g(w)}{(1-|w|^2)^n}\intd m(w)\\
			&+\int_{\bn\times\bn}\rho_{n,t}(|\varphi_z(w)|^2)L_zf(z-w)L_wg(z-w)\frac{\intd m(z,w)}{|1-\la z,w\ra|^{2n+2}},
		\end{flalign*}
		and 
		\begin{equation*}
			\lim_{t\to\infty}t^{1-n}\Tr\bigg(\BTt_f\BTt_g-\BTt_{fg}\bigg)=\frac{-1}{\pi^n}\int_{\bn}\frac{\sum_{i=1}^n\partial_if(w)\bpartial_ig(w)-Rf(w)\bar{R}g(w)}{(1-|w|^2)^n}\intd m(w).
		\end{equation*}
	\end{rem}

\begin{rem}
It is clear that if we have $\BHt_{\bar{f}}\in\mathcal{S}^p$ and $\BHt_g\in\mathcal{S}^q$ for some $\frac{1}{p}+\frac{1}{q}=1$ then the semi-commutator $\BTt_f\BTt_g-\BTt_{fg}$ belongs to the trace class. However, the converse is not true. 
The point of Condition 1 and Condition 2 is to give combined conditions of $f$ and $g$, instead of separate conditions. The estimates 
\begin{equation*}
	|1-\la z,w\ra|\approx(1-|z|^2)+(1-|w|^2)+|z-w|^2+|\mathrm{Im}\la z,w\ra|,\quad|\varphi_z(w)|^2=\frac{|z-P_z(w)|^2+(1-|z|^2)|Q_z(w)|^2}{|1-\la z,w\ra|^2}
\end{equation*}
give us some insight into the two conditions. In Lemmas \ref{lem: gradient product} - \ref{lem: f g compact support} we give some special cases of Theorem \ref{thm: higher dimensions} that are more intuitive and more convenient to work with. 
\end{rem}
	
	Our proofs involve applying integration by parts formulas on the unit disk and unit ball (see Lemma \ref{lem: formula D z} and Lemma \ref{lem: formula Bn}). These formulas essentially come from the Cauchy formula and a Bochner-Martinelli type formula we develop in Appendix I.
	 In Section \ref{sec: integration by parts}, integration by parts formulas on $\bn$ are developed. The formulas involve auxiliary functions and operations, which we define and study in Appendix II. In Section \ref{sec: the higher dimensions}, we prove Theorem \ref{thm: higher dimensions}. Some applications and examples are given in Section \ref{sec: applications and examples}.

	We end the introduction with some explanation on the relationships between this paper and our other paper \cite{Ta-Wa-Zh:HHtrace}. Our study was motivated by the exploration of the Helton-Howe trace and Connes-Chern character in \cite{Ta-Wa-Zh:HHtrace}, which is an important invariant in noncommutative differential geometry. In \cite{Ta-Wa-Zh:HHtrace}, we study the Helton-Howe trace and the Connes-Chern character for Toeplitz operators on weighted Bergman spaces via the idea of quantization, \cite{Be:quantization, Bo-Gu:spectral, Bo-Le-Ta-We:asymptotic, Co:deformation, En:asymptotics, En:berezin, En:forelli-rudin, En:weighted}. As a remainder term in the Toeplitz quantization, semi-commutators naturally appears in the proofs. On the other hand, many of the tools developed here are also heavily used in \cite{Ta-Wa-Zh:HHtrace}. The proofs in this paper are intended to be self-contained.\\

\noindent{\bf Acknowledgment:} We would like to thank Mohammad Jabbari, Richard Rochberg, Jingbo Xia and Kai Wang for inspiring discussions. Tang is partially supported by NSF Grants DMS 1800666, 1952551.

	\section{Preliminaries}\label{sec: preliminary}
	In this section, we recall some basic definitions and properties about weighted Bergman spaces and Schatten-$p$ class operators.
	
	Let $\bn$ be the open unit ball of $\cn$ and $\sn=\partial\bn$ the unit sphere. Let $m$ be the Lebesgue measure on $\bn$ and $\sigma$ be the surface measure on $\sn$. Denote $\sigma_{2n-1}=\sigma(\sn)=\frac{2\pi^n}{(n-1)!}$.

For $t>-1$, define the probability measure on $\bn$:
	\[
	\intd\lambda_t(z)=\frac{(n-1)!}{\pi^nB(n, t+1)}(1-|z|^2)^t\intd m(z).
	\]
	Here $B(n, t+1)$ is the Beta function. The weighted Bergman space $\bert$ is the subspace of $L^2(\bn, \lambda_t)$ consisting of holomorphic functions on $\bn$. The reproducing kernel of $\bert$ is
	\[
	\BKt_w(z)=\frac{1}{(1-\la z,w\ra)^{n+1+t}},\quad\forall w\in\bn.
	\]
	For any $f\in L^\infty(\bn)$, the Toeplitz operator $\BTt_f$ is the compression
	\[
	\BTt_f=\BPt M^{(t)}_f|_{\bert},
	\]
	where $\BPt$ is the orthogonal projection from $L^2(\bn, \lambda_t)$ onto $\bert$, and $M^{(t)}_f$ is the multiplication operator on $L^2(\bn,\lambda_t)$. The Hankel operator with symbol $f$ is
	\[
	H^{(t)}_f=(I-P^{(t)})M_f^{(t)}P^{(t)}.
	\]
	Using the reproducing kernels, we can write $\BTt_f, H^{(t)}_f$ as integral operators. For $h\in\bert$, we have the following expressions,
	\[
	\BTt_fh(z)=\int_{\bn}f(w)h(w)\BKt_w(z)\intd\lambda_t(w),\quad\forall z\in\bn.
	\]
	\[
	H^{(t)}_fh(z)=\int_{\bn}\big(f(z)-f(w)\big)h(w)\BKt_w(z)\intd\lambda_t(w),\quad\forall z\in\bn.
	\]

	An important tool on $\bn$ is the M\"{o}bius transform.
	\begin{defn}\label{defn: Mobius transform}
		For $z\in\bn$, $z\neq0$, the M\"{o}bius transform $\varphi_z$ is the biholomorphic mapping on $\bn$ defined as follows.
		\[
		\varphi_z(w)=\frac{z-P_z(w)-(1-|z|^2)^{1/2}Q_z(w)}{1-\la w,z\ra},\quad\forall w\in\overline{\bn}.
		\]
		Here $P_z$ and $Q_z$ denote the orthogonal projection from $\cn$ onto $\mathbb{C}z$ and $z^\perp$, respectively.
		Define
		\[
		\varphi_0(w)=-w,\quad\forall w\in\overline{\bn}.
		\]
		It is well-known that $\varphi_z$ is an automorphism of $\bn$ satisfying $\varphi_z\circ\varphi_z=Id$. Also, the two variable function $\rho(z,w):=|\varphi_z(w)|=|\varphi_w(z)|$ defines a metric on $\bn$. Moreover, $\beta(z,w):=\tanh^{-1}\rho(z,w)$ coincides with the Bergman metric on $\bn$.
	\end{defn}

	We list some lemmas that serve as basic tools for our study of Toeplitz operators on $\bn$. Most of the following of this section can be found in \cite{Rudinbookunitball, Zhubookspaces}. A proof will be provided when necessary.
	
	For non-negative values $A, B$, by $A\lesssim B$ we mean that there is a constant $C$ such that $A\leq CB$. Sometimes, to emphasize that the constant $C$ depends on some parameter $a$, we write $A\lesssim_a B$. The notations $\gtrsim, \gtrsim_a, \approx, \approx_a$ are defined similarly.

	\begin{lem}\label{lem: Mobius basics}
		Suppose $z, w, \zeta\in\bn$.
		\begin{itemize}
			\item[(1)] $\frac{1}{1-\la\varphi_{\zeta}(z),\varphi_{\zeta}(w)\ra}=\frac{(1-\la z,\zeta\ra)(1-\la\zeta,w\ra)}{(1-|\zeta|^2)(1-\la z,w\ra)}$.
			\item[(2)] $1-|\varphi_z(w)|^2=\frac{(1-|z|^2)(1-|w|^2)}{|1-\la z,w\ra|^2}$.
			\item[(3)] For any $R>0$ there exists $C>1$ such that whenever $\beta(z,w)<R$,
			\[
			\frac{1}{C}\leq\frac{1-|z|^2}{1-|w|^2}\leq C,\quad\frac{1}{C}\leq\frac{|1-\la z,\zeta\ra|}{|1-\la w,\zeta\ra|}\leq C.
			\]
			\item[(4)] The real Jacobian of $\varphi_z$ is $\frac{(1-|z|^2)^{n+1}}{|1-\la z,\cdot\ra|^{2n+2}}$ on $\bn$ and $\frac{(1-|z|^2)^n}{|1-\la z,\cdot\ra|^{2n}}$ on $\sn$.
			\item[(5)] For $z\in\bn$,
			\[
			z-\varphi_z(w)=\frac{(1-|z|^2)P_z(w)+(1-|z|^2)^{1/2}Q_z(w)}{1-\la w,z\ra}:=\frac{A_zw}{1-\la w,z\ra},
			\]
			where $A_z=[a_z^{ij}]$ is an $n\times n$ matrix depending on $z$, and $w$ is viewed as a column vector.
			\item[(6)] There exists $C>0$ such that for any $z\in\bn$, $z\neq0$,
			\begin{equation}\label{eqn: |z-w|}
				|z-P_z(w)|\leq |\varphi_z(w)||1-\la z,w\ra|,\quad |Q_z(w)|\leq C|\varphi_z(w)||1-\la z,w\ra|^{1/2},
			\end{equation}
			and
			\begin{equation}\label{eqn: tangent nontangent estimates}
				|z-w|\leq C|\varphi_z(w)||1-\la z,w\ra|^{1/2}.
			\end{equation}
			In contrast, if $n=1$, then $|z-w|=|\varphi_z(w)||1-z\bw|$.
			\item[(7)] $1-|z|^2\leq2|1-\la z,w\ra|$.
		\end{itemize}
	\end{lem}
	
	\begin{proof}
		Most of the above are either well-known (cf. \cite{Rudinbookunitball, Zhubookspaces}) or straightforward to verify. The only part that requires some clarification is the second estimate in (6), i.e.,
		\begin{equation}\label{eqn: temp 12}
			|Q_z(w)|\lesssim|\varphi_z(w)||1-\la z,w\ra|^{1/2}.
		\end{equation}
		On the one hand, from the definition of $\varphi_z(w)$, we easily get the follow inequality,
		\[
		|Q_z(w)|\leq|\varphi_z(w)|\frac{|1-\la z,w\ra|}{(1-|z|^2)^{1/2}}.
		\]
		If $|\varphi_z(w)|\leq\frac{1}{2}$, then by (3), $|1-\la z,w\ra|\approx 1-|z|^2$. From this \eqref{eqn: temp 12} follows.
		
		On the other hand, since
		\[
		2|1-\la z,w\ra|\geq 2-2\mathrm{Re}\la z,w\ra\geq|z|^2+|w|^2-2\mathrm{Re}\la z,w\ra=|z-w|^2,
		\]
		for $|\varphi_z(w)|>\frac{1}{2}$, we obtain the following estimates,
		\[
		|Q_z(w)|=|Q_z(w-z)|\leq|z-w|\lesssim |1-\la z,w\ra|^{\frac{1}{2}}<2|\varphi_z(w)||1-\la z,w\ra|^{\frac{1}{2}}.
		\]
		Thus we get \eqref{eqn: temp 12} in both cases. This completes the proof of Lemma \ref{lem: Mobius basics}.
	\end{proof}
	
	\begin{lem}(\cite[Proposition 5.1.2]{Rudinbookunitball})\label{lem: d metric}
		The two variable function $d(z,w)=|1-\la z, w\ra|^{\frac{1}{2}}$ on $\overline{\bn}$ satisfies the triangle inequality, i.e.,
		\[
		d(z,w)\leq d(z,\xi)+d(\xi,w),\quad\forall z, w, \xi\in\overline{\bn}.
		\]
	\end{lem}
	
	\begin{lem}[Rudin-Forelli type estimates]\label{lem: Rudin Forelli generalizations}~
		
		\begin{itemize}
			\item[(1)] Suppose $t>-1$, $c\in\mathbb{R}$. Then there exists $C>0$ such that for any $z\in\bn$,
			\begin{equation}\label{eqn: Rudin-Forelli 1-1}
				\int_{\bn}\frac{(1-|w|^2)^t}{|1-\la z,w\ra|^{n+1+t+c}}\intd m(w)\leq\begin{cases}
					C(1-|z|^2)^{-c},&c>0,\\
					C\ln\frac{1}{1-|z|^2},&c=0,\\
					C,&c<0,
				\end{cases}
			\end{equation}
			\begin{equation}\label{eqn: Rudin-Forelli 1-2}
				\int_{\sn}\frac{1}{|1-\la z,w\ra|^{n+c}}\intd\sigma(w)\leq \begin{cases}
					C(1-|z|^2)^{-c},&c>0,\\
					C\ln\frac{1}{1-|z|^2},&c=0,\\
					C,&c<0.
				\end{cases}
			\end{equation}
			\item[(2)] Suppose $t>-1$, $a, b, c>0$, $a\geq c, b\geq c$, and $a+b<n+1+t+c$. Then there exists $C>0$ such that for any $z, \xi\in\bn$,
			\begin{equation}\label{eqn: Rudin-Forelli 2}
				\int_{\bn}\frac{(1-|w|^2)^t}{|1-\la z, w\ra|^a|1-\la w, \xi\ra|^b}\intd m(w)\leq C\frac{1}{|1-\la z,\xi\ra|^c}.
			\end{equation}
			\item[(3)] Suppose $\phi:(0,1)\to[0,\infty)$ is measurable. Suppose $a>-n$, $b\in\mathbb{R}$, and
			\[
			\phi(s)\lesssim s^a(1-s)^b,\quad s\in(0,1).
			\]
			Then for any $t>-1-b$, $c>-b$ there exists $C>0$ such that for any $z\in\bn$,
			\begin{equation}\label{eqn: Rudin-Forelli 3}
				\int_{\bn}\phi(|\varphi_z(w)|^2)\frac{(1-|w|^2)^t}{|1-\la z,w\ra|^{n+1+t+c}}\intd m(w)\leq C(1-|z|^2)^{-c}.
			\end{equation}
		\end{itemize}
	\end{lem}
	
	\begin{proof}
		The estimates in (1) are standard Rudin-Forelli estimates. See \cite[Proposition 1.4.10]{Rudinbookunitball} for a proof.
		Let
		\[
		A=\{w\in\bn: |1-\la z,w\ra|\leq|1-\la w,\xi|\};\quad B=\{w\in\bn: |1-\la z,w\ra|>|1-\la w,\xi\ra|\}.
		\]
		By Lemma \ref{lem: d metric}, we have the following equality,
		\[
		|1-\la z,\xi\ra|^{1/2}\leq|1-\la z,w\ra|^{1/2}+|1-\la w,\xi\ra|^{1/2}.
		\]
		Then we obtain the following bounds,
		\[
		|1-\la z,w\ra|\geq \frac{1}{4}|1-\la z,\xi\ra|,\forall w\in B;\quad |1-\la w,\xi\ra|\geq\frac{1}{4}|1-\la z,\xi\ra|,\forall w\in A.
		\]
		By assumption, $a\geq c$, $b\geq c$, $a+b-c<n+1+t$, therefore by the standard Rudin-Forelli estimate, we compute the integral as follows,
		\begin{flalign*}
			&\int_{\bn}\frac{(1-|w|^2)^t}{|1-\la z,w\ra|^a|1-\la w,\xi\ra|^b}\intd m(w)\\
			=&\int_{A}\frac{(1-|w|^2)^t}{|1-\la z,w\ra|^a|1-\la w,\xi\ra|^b}\intd m(w)+\int_{B}\frac{(1-|w|^2)^t}{|1-\la z,w\ra|^a|1-\la w,\xi\ra|^b}\intd m(w)\\
			\lesssim&\frac{1}{|1-\la z,\xi\ra|^c}\int_{A}\frac{(1-|w|^2)^t}{|1-\la z,w\ra|^{a+b-c}}\intd m(w)+\frac{1}{|1-\la z,\xi\ra|^c}\int_{B}\frac{(1-|w|^2)^t}{|1-\la w,\xi\ra|^{b+a-c}}\intd m(w)\\
			\lesssim&\frac{1}{|1-\la z,\xi\ra|^c},
		\end{flalign*}
		This proves (2).
		
		To prove (3), make the change of variable $w=\varphi_z(\xi)$ in the left hand side of \eqref{eqn: Rudin-Forelli 3}. We compute the integral as follows,
		\begin{flalign*}
			&\int_{\bn}\phi(|\varphi_z(w)|^2)\frac{(1-|w|^2)^t}{|1-\la z,w\ra|^{n+1+t+c}}\intd m(w)\\
			=&\frac{1}{(1-|z|^2)^c}\int_{\bn}\phi(|\xi|^2)\frac{(1-|\xi|^2)^t}{|1-\la z,\xi\ra|^{n+1+t-c}}\intd m(\xi)\\
			=&\frac{1}{(1-|z|^2)^c}\int_0^1\phi(r^2)r^{2n-1}(1-r^2)^t\int_{\sn}\frac{1}{|1-\la z,r\eta\ra|^{n+1+t-c}}\intd \sigma(\eta)\intd r\\
			\lesssim&\frac{1}{(1-|z|^2)^c}\int_0^1\phi(r^2)r^{2n-1}(1-r^2)^m\intd r,
		\end{flalign*}
		where $m=c-1$ when $1+t-c>0$, $m=t$ when $1+t-c<0$, and when $1+t-c=0$, we take $m=t-\epsilon$ for a sufficiently small $\epsilon>0$. With our assumption it is easy to see that $m>-b-1$ and therefore the integral above is finite. This proves (3) and completes the proof of Lemma \ref{lem: d metric}.
	\end{proof}

	For $p>0$, a bounded operator $T$ on a Hilbert space $\mathcal{H}$ is said to be in the Schatten-$p$ class $\mathcal{S}^p$ if $|T|^p$ belongs to the trace class. The Schatten-$p$ class operators $\mathcal{S}^p$ are analogues of $l^p$ spaces in the operator-theoretic setting and satisfy the H\"{o}lder's inequality (see \cite[Theorem 2.8]{Simon}).
	
	The following lemma is well-known. See \cite[Theorem 6.4]{Zhu:bookoperator} for a proof at $n=1$. The same proof works for general $n$.
	\begin{lem}\label{lem: trace is integral of Berezin}
		Suppose $t>-1$ and $T$ is a trace class operator on $\bert$. Then
		\[
		\Tr T=\int_{\bn}\la T\BKt_z,\BKt_z\ra\intd\lambda_t(z).
		\]
	\end{lem}

	\section{Trace Formulas on the Disk}\label{sec: trace formulas on the disk}
	
	In this section we give the proof of Theorem \ref{thm: semi-commutator trace in dim 1}. The main ingredient of its proof is an integral formula coming from the Cauchy formula.
	
	\begin{defn}\label{defn: FG on disk}
		Suppose $t>-1$ and $\phi:(0,1)\to[0,\infty)$ is measurable. Define the operations on $\phi$:
		\[
		\BFt\phi(s)=\int_s^1\phi(r)(1-r)^t\intd r,\quad \BGt\phi(s)=s^{-1}(1-s)^{-t-1}\BFt\phi(s).
		\]
	\end{defn}
	
	\begin{lem}\label{lem: formula D 0}
		Suppose that $t>-1$, and $\phi:(0,1)\mapsto[0,\infty)$ is a measurable function, and $v\in\mathscr{C}^1(\overline{\dd})$. Assume that the two integrals
		\[
		\int_{\dd}\phi(|z|^2)v(z)\intd\lambda_t(z),\quad\int_{\dd}(1-|z|^2)\BGt\phi(|z|^2)\bz\bpartial v(z)\intd\lambda_t(z)
		\]
		both converge absolutely. Then
		\begin{flalign}\label{eqn: formula D 0}
			&\int_{\dd}\phi(|z|^2)v(z)\intd\lambda_t(z)
			\\
			=&\begin{cases}
				(t+1)\BFt\phi(0)\cdot v(0)+\int_{\dd}(1-|z|^2)\BGt\phi(|z|^2)\bz\bpartial v(z)\intd\lambda_t(z),&v(0)\neq0, \BFt\phi(0)<\infty,\\
				\int_{\dd}(1-|z|^2)\BGt\phi(|z|^2)\bz\bpartial v(z)\intd\lambda_t(z),&v(0)=0, \BFt\phi(0)\leq\infty.
			\end{cases}\nonumber
		\end{flalign}
	\end{lem}
	
	\begin{proof}
		Assume first that $v(0)=0$.
		For any $0<r<1$, denote $\sigma_r$ the Euclidean surface measure on $r\tori$. If we apply the Cauchy Formula to $\Omega=r\dd\subset\CC$ and $v\in\mathscr{C}^1(\overline{\dd})$, then we have the following equation of integrals,
		\begin{equation*}
			\int_{r\tori}v(z)\intd\sigma_r(z)=2r\int_{r\dd}\frac{\bpartial v(z)}{z}\intd m(z).
		\end{equation*}
		By assumption the left hand side of \eqref{eqn: formula D 0} is absolutely integrable, so we compute the integral as follows,
		\begin{flalign*}
			\int_{\dd}\phi(|z|^2)v(z)\intd\lambda_t(z)
			=&\frac{(t+1)}{\pi}\int_0^1\phi(r^2)(1-r^2)^t\bigg\{\int_{r\tori}v(z)\intd\sigma_r(z)\bigg\}\intd r\\
			=&\frac{(t+1)}{\pi}\int_0^1\phi(r^2)(1-r^2)^t\bigg\{2r\int_{r\dd}\frac{\bpartial v(z)}{z}\intd m(z)\bigg\}\intd r\\
			=&\frac{(t+1)}{\pi}\int_{\dd}\int_{|z|}^12r\phi(r^2)(1-r^2)^t\intd r\cdot\frac{\bpartial v(z)}{z}\intd m(z)\\
			=&\frac{(t+1)}{\pi}\int_{\dd}\bigg[\int_{|z|^2}^1\phi(s)(1-s)^t\intd s\bigg]\frac{\bpartial v(z)}{z}\intd m(z)\\
			=&\int_{\dd}(1-|z|^2)\BGt\phi(|z|^2)\bz\bpartial v(z)\intd\lambda_t(z).
		\end{flalign*}
		The absolute convergence of the integral in the right hand side of \eqref{eqn: formula D 0} ensures the third equality above.
		This proves the second case.
		
		Now assume that $v(0)\neq0$, and $\BFt\phi(0)<\infty$. We notice that
		\[
		\int_{\dd}\phi(|z|^2)\intd\lambda_t(z)=2\pi\cdot\frac{t+1}{\pi}\int_0^1r\phi(r^2)(1-r^2)^t\intd r=(t+1)\BFt\phi(0).
		\]
		Then we get the following estimate
		\[
		\int_{\dd}\bigg|\phi(|z|^2)\big(v(z)-v(0)\big)\bigg|\intd\lambda_t(z)\leq\int_{\dd}\bigg|\phi(|z|^2)v(z)\bigg|\intd\lambda_t(z)+(t+1)\BFt\phi(0)|v(0)|<\infty.
		\]
		Applying the second case to $v(z)-v(0)$ and reorganizing the terms give the first case. This completes the proof of Lemma \ref{lem: formula D 0}.
	\end{proof}

	\begin{lem}\label{lem: formula D z}
		Suppose that $t>-1$, and $\phi: (0,1)\to[0,\infty)$ is a measurable function, and $v\in\mathscr{C}^1(\overline{\dd})$.
		\begin{itemize}
			\item[(1)] Assuming that $z\in\dd$, and the integrals
			\begin{eqnarray*}
				&\int_{\dd}\phi(|\varphi_z(w)|^2)v(w)\BKt_w(z)\intd\lambda_t(w),\\
				&\int_{\dd}\BGt\phi(|\varphi_z(w)|^2)\frac{(1-|w|^2)\overline{(z-w)}}{1-w\bz}\bpartial v(w)\BKt_w(z)\intd\lambda_t(w)
			\end{eqnarray*}
			converge absolutely, then
			\begin{flalign}\label{eqn: formula D z d''}
				&\int_{\dd}\phi(|\varphi_z(w)|^2)v(w)\BKt_w(z)\intd\lambda_t(w)\\
				=&\begin{cases}
					(t+1)\BFt\phi(0)\cdot v(z)&v(z)\neq0, \BFt\phi(0)<\infty,\\
					~~-\int_{\dd}\BGt\phi(|\varphi_z(w)|^2)\frac{(1-|w|^2)\overline{(z-w)}}{1-w\bz}\bpartial v(w)\BKt_w(z)\intd\lambda_t(w),\\
					\\
					-\int_{\dd}\BGt\phi(|\varphi_z(w)|^2)\frac{(1-|w|^2)\overline{(z-w)}}{1-w\bz}\bpartial v(w)\BKt_w(z)\intd\lambda_t(w),&v(z)=0,\BFt\phi(0)\leq\infty.
				\end{cases}\nonumber
			\end{flalign}
			\item[(2)] Assuming that $w\in\dd$, and the integrals
			\begin{eqnarray*}
				\int_{\dd}\phi(|\varphi_z(w)|^2)v(z)\BKt_w(z)\intd\lambda_t(z),\qquad \int_{\dd}\BGt\phi(|\varphi_z(w)|^2)\frac{(1-|z|^2)(z-w)}{1-w\bz}\partial v(z)\BKt_w(z)\intd\lambda_t(z)
			\end{eqnarray*}
			converge absolutely. Then
			\begin{flalign}\label{eqn: formula D z d'}
				&\int_{\dd}\phi(|\varphi_z(w)|^2)v(z)\BKt_w(z)\intd\lambda_t(z)\\
				=&\begin{cases}
					(t+1)\BFt\phi(0)\cdot v(w)&v(w)\neq0, \BFt\phi(0)<\infty,\\
					~~+\int_{\dd}\BGt\phi(|\varphi_z(w)|^2)\frac{(1-|z|^2)(z-w)}{1-w\bz}\partial v(z)\BKt_w(z)\intd\lambda_t(z),\\
					\\
					\int_{\dd}\BGt\phi(|\varphi_z(w)|^2)\frac{(1-|z|^2)(z-w)}{1-w\bz}\partial v(z)\BKt_w(z)\intd\lambda_t(z),&v(w)=0, \BFt\phi(0)\leq\infty.
				\end{cases}\nonumber
			\end{flalign}
		\end{itemize}

	\end{lem}
	
	\begin{proof} First, we prove case (1). The formula is obtained from \eqref{eqn: formula D 0} by taking M\"{o}bius transforms. By Lemma \ref{lem: Mobius basics} (4) it is easy to verify the following equation,
		\begin{equation}\label{eqn: Mobius and kernel}
			\BKt_w(z)\intd\lambda_t(w)\xlongequal[\xi=\varphi_z(w)]{w=\varphi_z(\xi)}\BKt_z(\xi)\intd\lambda_t(\xi).
		\end{equation}	
		By Lemma \ref{lem: Mobius basics} (1)(2), Lemma \ref{lem: formula D 0} and the above equality, if $\BFt\phi(0)<\infty$, then we compute the integral as follows.
		\begin{flalign*}
			&\int_{\dd}\phi(|\varphi_z(w)|^2)v(w)\BKt_w(z)\intd\lambda_t(w)\\
			\xlongequal{w=\varphi_z(\xi)}&\int_{\dd}\phi(|\xi|^2)v\circ\varphi_z(\xi)\BKt_z(\xi)\intd\lambda_t(\xi)\\
			\xlongequal{\eqref{eqn: formula D 0}}&(t+1)\BFt\phi(0)\cdot v(z)+\int_{\dd}\BGt\phi(|\xi|^2)(1-|\xi|^2)\bar{\xi}\bpartial_\xi\bigg(v\circ\varphi_z(\xi)\BKt_z(\xi)\bigg)\intd\lambda_t(\xi)\\
			=&(t+1)\BFt\phi(0)\cdot v(z)+\int_{\dd}\BGt\phi(|\xi|^2)(1-|\xi|^2)\bar{\xi}\bpartial v(\varphi_z(\xi))\overline{\varphi_z'(\xi)}\BKt_z(\xi)\intd\lambda_t(\xi)\\
			=&(t+1)\BFt\phi(0)\cdot v(z)-\int_{\dd}\BGt\phi(|\xi|^2)(1-|\xi|^2)\bar{\xi}\bpartial v(\varphi_z(\xi))\frac{1-|z|^2}{(1-z\bar{\xi})^2}\BKt_z(\xi)\intd\lambda_t(\xi)\\
			\xlongequal{\xi=\varphi_z(w)}&(t+1)\BFt\phi(0)\cdot v(z)-\int_{\dd}\BGt\phi(|\varphi_z(w)|^2)\frac{(1-|w|^2)\overline{(z-w)}}{1-w\bz}\bpartial v(w)\BKt_w(z)\intd\lambda_t(w).
		\end{flalign*}
		The case when $\BFt\phi(0)=\infty, v(z)=0$ is proved by the same equations as above, but with the term ``$(t+1)\BFt\phi(0)\cdot v(z)$'' removed. This proves (1). To prove (2), apply \eqref{eqn: formula D z d''} to $\overline{v(z)}$, then swap $z$ and $w$, then take conjugate over the equation. Note that the equations $|\varphi_z(w)|=|\varphi_w(z)|$ and $\overline{\BKt_z(w)}=\BKt_w(z)$ is used here. This completes the proof of Lemma \ref{lem: formula D z}.
	\end{proof}
	Recall that in Theorem \ref{thm: semi-commutator trace in dim 1} we defined
	\[
	F(s,x)=-\bigg[x\ln\frac{s}{x}+(1-x)\ln\frac{1-s}{1-x}\bigg].
	\]
	\begin{lem}\label{lem: F(s,x) as integral}
		Suppose $0<s<x<1$. Then
		\[
		\iint_{s<s_1<s_2<x}s_1^{-1}(1-s_1)^{-1}\intd s_1\intd s_2=F(s,x)
		\]
	\end{lem}
	\begin{proof}
		\begin{flalign*}
			&\iint_{s<s_1<s_2<x}s_1^{-1}(1-s_1)^{-1}\intd s_1\intd s_2\\
			=&\int_s^x\int_{s_1}^xs_1^{-1}(1-s_1)^{-1}\intd s_2\intd s_1\\
			=&\int_s^x\frac{x-s_1}{s_1(1-s_1)}\intd s_1\\
			=&\int_s^x\frac{x}{s_1}-\frac{1-x}{1-s_1}\intd s_1\\
			=&x\ln\frac{x}{s}+(1-x)\ln\frac{1-x}{1-s}\\
			=&F(s,x).
		\end{flalign*}
		This completes the proof.
	\end{proof}
	
	\begin{lem}\label{lem: F(s,x) integral estimate}
		For $0<\epsilon<1$ there exists $C>0$ such that
		\[
		\int_0^x(1-s)^{-\epsilon} F(s,x)\intd s\leq Cx^2,\quad0<x<1.
		\]
	\end{lem}
	
	\begin{proof}
		By definition,
		\begin{flalign*}
			&\int_0^x(1-s)^{-\epsilon} F(s,x)\intd s\\
			=&-x\int_0^x(1-s)^{-\epsilon} \ln\frac{s}{x}\intd s-(1-x)\int_0^x(1-s)^{-\epsilon} \ln(1-s)\intd s+(1-x)\ln (1-x)\int_0^x(1-s)^{-\epsilon}\intd s.
		\end{flalign*}
		For $0<x<1$, $\ln(1-x)<0$. Thus the last term in the above is negative. Therefore
		\begin{equation}\label{eqn: proof F(s,x) int}
			\int_0^x(1-s)^{-\epsilon} F(s,x)\intd s
			<-x\int_0^x(1-s)^{-\epsilon} \ln\frac{s}{x}\intd s-(1-x)\int_0^x(1-s)^{-\epsilon} \ln(1-s)\intd s.
		\end{equation}
		For the first term in the right hand side of \eqref{eqn: proof F(s,x) int}, take the change of variable $r=\frac{s}{x}$. Then
		\begin{equation}\label{eqn: temp 1}
			-x\int_0^x(1-s)^{-\epsilon}\ln\frac{s}{x}\intd s\xlongequal{r=s/x}-x^2\int_0^1(1-rx)^{-\epsilon}\ln r\intd r<-x^2\int_0^1(1-r)^{-\epsilon}\ln r\intd r\leq x^2.
		\end{equation}
		For the second term in the right hand side of \eqref{eqn: proof F(s,x) int}, notice that
		\[
		-\ln(1-s)\lesssim s.
		\]
		So
		\begin{equation}\label{eqn: temp 2}
			-(1-x)\int_0^x(1-s)^{-\epsilon} s\intd s\leq (1-x)^{1-\epsilon}\int_0^xs\intd s\lesssim x^2.
		\end{equation}
		Combining \eqref{eqn: proof F(s,x) int}, \eqref{eqn: temp 1} and \eqref{eqn: temp 2} gives the desired result.
	\end{proof}

	\begin{proof}[{\bf Proof of Theorem \ref{thm: semi-commutator trace in dim 1}}]
		By Lemma \ref{lem: trace is integral of Berezin}, we have the following expression of $\Tr\bigg(\BTt_f\BTt_g-\BTt_{fg}\bigg)$,
		\begin{flalign*}
			&\Tr\bigg(\BTt_f\BTt_g-\BTt_{fg}\bigg)=\int_{\dd}\la\bigg(\BTt_f\BTt_g-\BTt_{fg}\bigg)\BKt_\xi,\BKt_\xi\ra\intd\lambda_t(\xi)\\
			=&\int_{\dd}\bigg\{\int_{\dd^2}(f(z)-f(w))g(w)\BKt_\xi(w)\BKt_w(z)\BKt_z(\xi)\intd\lambda_t(w)\intd\lambda_t(z)\bigg\}\intd\lambda_t(\xi).
		\end{flalign*}
		The rest of the proof is simply iterating \eqref{eqn: formula D z d''} and \eqref{eqn: formula D z d'} on the integral above. For each fixed $\xi\in\dd$, we calculate the inner product $\la\bigg(\BTt_f\BTt_g-\BTt_{fg}\bigg)\BKt_\xi,\BKt_\xi\ra$ as follows,
		\begin{flalign*}
			&\la\bigg(\BTt_f\BTt_g-\BTt_{fg}\bigg)\BKt_\xi,\BKt_\xi\ra\\
			=&\int_{\dd^2}(f(z)-f(w))g(w)\BKt_\xi(w)\BKt_w(z)\BKt_z(\xi)\intd\lambda_t(w)\intd\lambda_t(z)\\
			=&\int_{\dd}\bigg\{\int_{\dd}\big[(f(z)-f(w))g(w)\BKt_\xi(w)\BKt_z(\xi)\big]\BKt_w(z)\intd\lambda_t(z)\bigg\}\intd\lambda_t(w)\\
			\xlongequal{\eqref{eqn: formula D z d'}}&\int_{\dd}\bigg\{\int_{\dd}\BGt1(|\varphi_z(w)|^2)\frac{(1-|z|^2)(z-w)}{1-w\bz}\big[\partial f(z)g(w)\BKt_\xi(w)\BKt_z(\xi)\big]\BKt_w(z)\intd\lambda_t(z)\bigg\}\intd\lambda_t(w).
		\end{flalign*}
		In the above, we apply \eqref{eqn: formula D z d'} with $\phi=1$ and $v(z)=\big[(f(z)-f(w))g(w)\BKt_\xi(w)\BKt_z(\xi)\big]$.
		Here, by direct computation, we obtain
		\[
		\big(\BGt1\big)(s)=\frac{1}{(t+1)s}.
		\]
		Since $\xi\in\dd$ is fixed, and $f, g$ are $\mathscr{C}^1$ to the boundary, the term $\big[\partial f(z)g(w)\BKt_\xi(w)\BKt_z(\xi)\big]$ is bounded.
		By \eqref{eqn: Rudin-Forelli 3}, the two-fold integral in the above converge absolutely. Applying Fubini's Theorem, and then \eqref{eqn: formula D z d''} with
		\[
		\phi=\BGt1,\quad v(w)=\frac{(1-|z|^2)(z-w)}{1-w\bz}\big[\partial f(z)g(w)\BKt_\xi(w)\BKt_z(\xi)\big],
		\]
		the above integral is computed as follows,
		\begin{flalign*}
			&\int_{\dd}\bigg\{\int_{\dd}\BGt1(|\varphi_z(w)|^2)\frac{(1-|z|^2)(z-w)}{1-w\bz}\big[\partial f(z)g(w)\BKt_\xi(w)\BKt_z(\xi)\big]\BKt_w(z)\intd\lambda_t(w)\bigg\}\intd\lambda_t(z)\\
			\xlongequal{\eqref{eqn: formula D z d''}}&-\int_{\dd}\bigg\{\int_{\dd}\big(\BGt\big)^21(|\varphi_z(w)|^2)\frac{(1-|w|^2)\overline{(z-w)}}{1-w\bz}\cdot\frac{(1-|z|^2)(z-w)}{1-w\bz}\\
			&\qquad \qquad \cdot\big[\partial f(z)\bpartial g(w)\BKt_\xi(w)\BKt_z(\xi)\big]\BKt_w(z)\intd\lambda_t(w)\bigg\}\intd\lambda_t(z)\\
			=&-\int_{\dd}\bigg\{\int_{\dd}\big(\BGt\big)^21(|\varphi_z(w)|^2)(1-|\varphi_z(w)|^2)|\varphi_z(w)|^2(1-z\bw)^2\\
			&\qquad \qquad \cdot\big[\partial f(z)\bpartial g(w)\BKt_\xi(w)\BKt_z(\xi)\big]\BKt_w(z)\intd\lambda_t(w)\bigg\}\intd\lambda_t(z)\\
			=&-\int_{\dd^2}\psi_1(|\varphi_z(w)|^2)(1-z\bw)^2\big[\partial f(z)\bpartial g(w)\BKt_\xi(w)\BKt_z(\xi)\big]\BKt_w(z)\intd\lambda_t(w)\intd\lambda_t(z).
		\end{flalign*}
		Here we have the following bound of $\psi_1$,
		\begin{equation}
			\psi_1(s)=\bigg[\big(\BGt\big)^21(s)\bigg](1-s)s=\frac{\int_s^1r^{-1}(1-r)^t\intd r}{(t+1)(1-s)^t}\lesssim_ts^{-1/2}(1-s).
		\end{equation}
		Using \eqref{eqn: Rudin-Forelli 2} and \eqref{eqn: Rudin-Forelli 3}, we can show that the integrand above is absolutely integrable with the measure $\intd\lambda_t(\xi)\intd\lambda_t(w)\intd\lambda_t(z)$. Therefore by Lemma \ref{lem: trace is integral of Berezin} and Fubini's Theorem, we compute $\Tr\bigg(\BTt_f\BTt_g-\BTt_{fg}\bigg)$ as follows,
		\begin{flalign*}
			&\Tr\bigg(\BTt_f\BTt_g-\BTt_{fg}\bigg)\\
			=&\int_{\dd}\bigg\{-\int_{\dd^2}\psi_1(|\varphi_z(w)|^2)(1-z\bw)^2\big[\partial f(z)\bpartial g(w)\BKt_\xi(w)\BKt_z(\xi)\big]\BKt_w(z)\intd\lambda_t(w)\intd\lambda_t(z)\bigg\}\intd\lambda_t(\xi)\\
			=&-\int_{\dd^2}\psi_1(|\varphi_z(w)|^2)(1-z\bw)^2\partial f(z)\bpartial g(w)|\BKt_w(z)|^2\intd\lambda_t(w)\intd\lambda_t(z).
		\end{flalign*}
		To obtain \eqref{eqn: semi-commutator trace in dim 1}, we apply \eqref{eqn: formula D z d'} again, with $z, w$ reversed,
		\[
		\phi=\psi_1,\qquad v(w)=\big[(1-z\bw)^2\partial f(z)\bpartial g(w)\BKt_w(z)\big].
		\]
		We compute the above integral as follows.
		\begin{flalign*}
			&\Tr\bigg(\BTt_f\BTt_g-\BTt_{fg}\bigg)\\
			=&-\int_{\dd}\bigg\{\int_{\dd}\psi_1(|\varphi_z(w)|^2)\big[(1-z\bw)^2\partial f(z)\bpartial g(w)\BKt_w(z)\big]\BKt_z(w)\intd\lambda_t(w)\bigg\}\intd\lambda_t(z)\\
			\xlongequal{\eqref{eqn: formula D z d'}}&-\int_{\dd}\bigg\{(t+1)\BFt\psi_1(0)\cdot(1-|z|^2)^2\partial f(z)\bpartial g(z)\BKt_z(z) \\
			&\hspace{1cm}+\int_{\dd}\BGt\psi_1(|\varphi_z(w)|^2)\frac{(1-|w|^2)(w-z)}{1-z\bw}\\
			&\qquad \qquad \qquad \qquad \big[(1-z\bw)^2\partial f(z)\partial\bpartial g(w)\BKt_w(z)\big]\BKt_z(w)\intd\lambda_t(w)\bigg\}\intd\lambda_t(z)\\
			=&\frac{1}{2\pi i}\int_{\dd}\partial f\wedge\bpartial g\\
			&\quad -\int_{\dd^2}\psi_2(|\varphi_z(w)|^2)\big[(1-|w|^2)(w-z)(1-z\bw)\partial f(z)\partial\bpartial g(w)\BKt_w(z)\big]\BKt_z(w)\intd\lambda_t(w)\intd\lambda_t(z).
		\end{flalign*}
		Here we have the following expressions
		\begin{equation}
			\BFt\psi_1(0)=\frac{\int_0^1\int_s^1r^{-1}(1-r)^t\intd r\intd s}{t+1}=\frac{\int_0^1\int_0^rr^{-1}(1-r)^t\intd s\intd r}{t+1}=\frac{\int_0^1(1-r)^t\intd r}{t+1}=(t+1)^{-2},
		\end{equation}
		and
		\begin{equation}
			\psi_2(s)=\BGt\psi_1(s)=\frac{\int_s^1\psi_1(r)(1-r)^t\intd r}{s(1-s)^{t+1}}=\frac{\int_s^1\int_r^1x^{-1}(1-x)^t\intd x\intd r}{(t+1)s(1-s)^{t+1}}\lesssim_ts^{-1/2}(1-s).
		\end{equation}
		Again, applying Fubini's Theorem and \eqref{eqn: formula D z d''}, with $z, w$ reversed, $\phi=\psi_2$,  $v(z)=\big[(1-|w|^2)(w-z)(1-z\bw)\partial f(z)\partial\bpartial g(w)\BKt_w(z)\big]$, we compute the second integral in the above  $\Tr\bigg(\BTt_f\BTt_g-\BTt_{fg}\bigg)$ as follows.
		\begin{flalign*}
			&-\int_{\dd^2}\psi_2(|\varphi_z(w)|^2)\big[(1-|w|^2)(w-z)(1-z\bw)\partial f(z)\partial\bpartial g(w)\BKt_w(z)\big]\BKt_z(w)\intd\lambda_t(w)\intd\lambda_t(z)\\
			=&-\int_{\dd} \bigg\{\int_{\dd}\psi_2(|\varphi_z(w)|^2)\big[(1-|w|^2)(w-z)(1-z\bw)\partial f(z)\partial\bpartial g(w)\BKt_w(z)\big]\BKt_z(w)\intd\lambda_t(z)\bigg\}\intd\lambda_t(w)\\
			\xlongequal{\eqref{eqn: formula D z d''}}&\int_{\dd}\bigg\{\int_{\dd}\BGt\psi_2(|\varphi_z(w)|^2)\frac{(1-|z|^2)\overline{(w-z)}}{1-z\bw}\\
			&~~\cdot\big[(1-|w|^2)(w-z)(1-z\bw)\bpartial\partial f(z)\partial\bpartial g(w)\BKt_w(z)\big]\BKt_z(w)\intd\lambda_t(z)\bigg\}\intd\lambda_t(w)\\
			=&\frac{(t+1)^2}{16\pi^2}\int_{\dd^2}\BGt\psi_2(|\varphi_z(w)|^2)\frac{(1-|z|^2)^{t+1}(1-|w|^2)^{t+1}|w-z|^2}{|1-z\bw|^{4+2t}}\Delta f(z)\Delta g(w)|\intd m(z,w)\\
			=&\int_{\dd^2}\varrho_t(|\varphi_z(w)|^2)\Delta f(z)\Delta g(w)\intd m(z,w).
		\end{flalign*}
		Here, by Lemma \ref{lem: Mobius basics}, we have
		\begin{flalign}\label{eqn: rho_t}
			\varrho_t(s)=&\frac{(t+1)^2}{16\pi^2}\BGt\psi_2(s)(1-s)^{t+1}s
			=\frac{(t+1)^2}{16\pi^2}\int_s^1\psi_2(s_1)(1-s_1)^t\intd s_1\nonumber\\
			=&\frac{(t+1)}{16\pi^2}\int_s^1\int_{s_1}^1\int_{s_2}^1s_1^{-1}(1-s_1)^{-1}s_3^{-1}(1-s_3)^t\intd s_3\intd s_2\intd s_1.
		\end{flalign}
		Therefore we have reached the following identity
		\begin{equation}\label{eqn: temp 9}
			\Tr\bigg(\BTt_f\BTt_g-\BTt_{fg}\bigg)=\frac{1}{2\pi i}\int_{\dd}\partial f\wedge\bpartial g+\int_{\dd^2}\varrho_t(|\varphi_z(w)|^2)\Delta f(z)\Delta g(w)\intd m(z,w).
		\end{equation}

		Let us simplify the expression of $\varrho_t$.
		\begin{flalign*}
			&\int_s^1\int_{s_1}^1\int_{s_2}^1s_1^{-1}(1-s_1)^{-1}s_3^{-1}(1-s_3)^t\intd s_3\intd s_2\intd s_1\\
			=&\iiint_{\{(s_1,s_2,s_3): s<s_1<s_2<s_3<1\}}s_1^{-1}(1-s_1)^{-1}s_3^{-1}(1-s_3)^t\intd s_1\intd s_2\intd s_3\\
			=&\int_s^1\bigg\{\iint_{\{(s_1,s_2): s<s_1<s_2<s_3\}}s_1^{-1}(1-s_1)^{-1}\intd s_1\intd s_2 \bigg\}s_3^{-1}(1-s_3)^t\intd s_3\\
			=&\int_s^1F(s,s_3)s_3^{-1}(1-s_3)^t\intd s_3.
		\end{flalign*}
		Here the last equality is by Lemma \ref{lem: F(s,x) as integral}. This proves the equation for $\varrho_t$. It also follows from Lemma \ref{lem: F(s,x) as integral} that $F(s,x)$ is strictly positive on $(0,1)$. Therefore $\rho_t$ is strictly positive on $(0,1)$.

		It remains to prove \eqref{eqn: semi-commutator trace in dim 1 limit}. In other words, the second term of \eqref{eqn: temp 9} tends to zero as $t$ tends to infinity.
		Clearly the absolute value of the second term has the following bound,
		\begin{flalign*}
			\lesssim&\int_{\dd}\int_{\dd}\varrho_t(|\varphi_z(w)|^2)\intd m(w)\intd m(z)
			=\int_{\dd}\int_{\dd}\varrho_t(|\zeta|^2)\frac{(1-|z|^2)^2}{|1-\zeta\bz|^4}\intd m(\zeta)\intd m(z)\\
			\lesssim&\int_{\dd}\varrho_t(|\zeta|^2)\ln\frac{1}{1-|\zeta|^2}\intd m(\zeta)
			\lesssim\int_{\dd}\varrho_t(|\zeta|)(1-|\zeta|^2)^{-1/2}\intd m(\zeta)\\
			\approx&\int_0^1\varrho_t(s)(1-s)^{-1/2}\intd s.
		\end{flalign*}
		Plugging in the formula of $\varrho_t$ and applying the Fubini's theorem gives
		\begin{flalign*}
			&\int_0^1\varrho_t(s)(1-s)^{-1/2}\intd s\\
			=&\frac{t+1}{16\pi^2}\int_0^1\int_s^1(1-s)^{-1/2}F(s,x)x^{-1}(1-x)^t\intd x\intd s\\
			=&\frac{t+1}{16\pi^2}\int_0^1\bigg[\int_0^x(1-s)^{-1/2}F(s,x)\intd s\bigg]x^{-1}(1-x)^t\intd x.
		\end{flalign*}
		By Lemma \ref{lem: F(s,x) integral estimate},
		\[
		0<\int_0^x(1-s)^{-1/2}F(s,x)\intd s\lesssim x^2.
		\] 
		So
		\begin{flalign*}
			\int_0^1\varrho_t(s)(1-s)^{-1/2}\intd s=&\frac{t+1}{16\pi^2}\int_0^1\bigg[\int_0^x(1-s)^{-1/2}F(s,x)\intd s\bigg]x^{-1}(1-x)^t\intd x\\
			\lesssim&(t+1)\int_0^1x(1-x)^t\intd x\\
			=&(t+1)B(2,t+1)\\
			\approx& t^{-1}.
		\end{flalign*}
		Therefore the second term in \eqref{eqn: temp 9} vanishes as $t\to\infty$. This completes the proof of Theorem \ref{thm: semi-commutator trace in dim 1}.
	\end{proof}

	\section{Integration by Parts}\label{sec: integration by parts}
	In the remaining sections of the article, we aim to extend Theorem \ref{thm: semi-commutator trace in dim 1} to higher dimensions and prove Theorem \ref{thm: higher dimensions}. Reviewing the proof of Theorem \ref{thm: semi-commutator trace in dim 1}, we notice that there are two key ingredients:
	\begin{itemize}
		\item[(1)] the integral formulas in Lemma \ref{lem: formula D z};
		\item[(2)] the auxiliary operations $\BFt$, $\BGt$ that record the change in $\phi$ after each application of the formulas.
	\end{itemize}
	With the above tools, we obtain the trace formula in Theorem \ref{thm: semi-commutator trace in dim 1} by applying iterations of Lemma \ref{lem: formula D z}.
	
	The proof of Theorem \ref{thm: higher dimensions} relies on generalizing (1) and (2). The goal of this section is to establish Lemma \ref{lem: formula Bn}, which is an analogue of Lemma \ref{lem: formula D z} in higher dimensions. Applying iteration of Lemma \ref{lem: formula Bn} two times, we get Lemma \ref{lem: formula F induction}. In Appendix II, more general auxiliary operations $\BFt_m$, $\BGt_m$ are defined and some basic properties are established.
	 In Section \ref{sec: the higher dimensions}, we apply Lemma \ref{lem: formula F induction} to obtain a formula for the semi-commutator (see Lemma \ref{lem: semicom = PR}), and we apply Lemma \ref{lem: formula F induction} again with $z, w$ reversed, to get the trace formula in Theorem \ref{thm: higher dimensions}. 
	 
	 The proof of Lemma \ref{lem: formula Bn} relies on a Bochner-Martinelli type formula that we establish in Appendix I. Similar integral formulas as in Lemma \ref{lem: formula Bn} were discovered by Charpentier in \cite{Charpentier}. Such integral formulas were used in \cite{OF2000} to study Bergman-Besov spaces and in \cite{Co-Sa-Wi:corona} to study the corona problem on the multiplier algebra of the Drury-Arveson space. 
	
	Let us start with a few definitions.
	Recall that by Lemma \ref{lem: Mobius basics} (5), for $z\in\bn$,
	\[
	(1-|z|^2)P_z(w)+(1-|z|^2)^{1/2}Q_z(w)=(1-\la w,z\ra)(z-\varphi_z(w))=A_zw,
	\]
	where $A_z$ is an $n\times n$ matrix depending on $z$, and $w$ is treated as a column vector.
	
	\begin{defn}\label{defn: d numbers I function}
		For multi-indices $\alpha, \beta\in\ind$ and $z\in\bn$, define
		\[
		d_{\alpha,\beta}(z)=\int_{\sn}\big(A_z\zeta\big)^\alpha\overline{\big(A_z\zeta\big)}^\beta\frac{\intd\sigma(\zeta)}{\sigma_{2n-1}}.
		\]
		In particular, $d_{0,0}=1$, and
		\begin{equation}\label{eqn: d at (1,0...0)}
			d_{\alpha,\beta}(z)=\delta_{\alpha,\beta}(1-|z|^2)^{\alpha_1+|\alpha|}\frac{(n-1)!\alpha!}{(n-1+|\alpha|)!},\quad\text{if }z=(z_1,0,\ldots,0).
		\end{equation}
		For multi-indices $\alpha, \beta\in\ind$ and $\zeta\in\cn$, denote
		\[
		I^{\alpha,\beta}(\zeta)=\zeta^\alpha\bar{\zeta}^\beta.
		\]
	\end{defn}
	\begin{lem}\label{lem: formula Bn}
		Suppose $t>-1$, $\alpha, \beta\in\ind$. Suppose $\phi: (0,1)\to[0,\infty)$ is measurable and $v\in\mathscr{C}^1(\bn)$. Then the following hold.
		\begin{enumerate}
			\item If $|\alpha|\geq|\beta|$ and all integrals converge absolutely, then
			\begin{flalign}\label{eqn: formula Bn d''}
				&\int_{\bn}\phi(|\varphi_z(w)|^2)I^{\alpha,\beta}(z-w) v(w)\BKt_w(z)\intd\lambda_t(w)\\
				=&\begin{cases}
					\frac{d_{\alpha,\beta}(z)}{B(n,t+1)}\cdot \BFt_{n+|\beta|}\phi(0) v(z)-\sum_{j=1}^n\int_{\bn}\BGt_{|\beta|+n}\phi(|\varphi_z(w)|^2)I^{\alpha,\beta+e_j}(z-w)S_j(w)\BKt_w(z)\intd\lambda_t(w),&\\
					\hspace{10cm} v(z)\neq0, \BFt_{n+|\beta|}\phi(0)<\infty,&\\
					&\\
					-\sum_{j=1}^n\int_{\bn}\BGt_{|\beta|+n}\phi(|\varphi_z(w)|^2)I^{\alpha,\beta+e_j}(z-w)S_j(w)\BKt_w(z)\intd\lambda_t(w),&\\
					\hspace{10cm}v(z)=0, \BFt_{n+|\beta|}\phi(0)\leq\infty,&
				\end{cases}
				\nonumber
			\end{flalign}
			where
			\[
			S_j(w,z)=\frac{(1-|w|^2)\bpartial_{w_j}\big[(1-\la z,w\ra)^{|\beta|}v(w)\big]}{(1-\la w,z\ra)(1-\la z,w\ra)^{|\beta|}}.
			\]
			\item If $|\alpha|\leq|\beta|$ and all integrals converge absolutely, then
			\begin{flalign}\label{eqn: formula Bn d'}
				&\int_{\bn}\phi(|\varphi_z(w)|^2)I^{\alpha,\beta}(z-w)v(z)\BKt_w(z)\intd\lambda_t(z)\\
				=&\begin{cases}
					\frac{d_{\alpha,\beta}(w)}{B(n,t+1)}\cdot\BFt_{n+|\alpha|}\phi(0)v(w)+\sum_{i=1}^n\int_{\bn}\BGt_{|\alpha|+n}\phi(|\varphi_z(w)|^2)I^{\alpha+e_i,\beta}(z-w)\widetilde{S}_i(z)\BKt_w(z)\intd\lambda_t(z),&\\
					\hspace{10cm}v(w)\neq0, \BFt_{n+|\alpha|}\phi(0)<\infty,&\\
					&\\
					\sum_{i=1}^n\int_{\bn}\BGt_{|\alpha|+n}\phi(|\varphi_z(w)|^2)I^{\alpha+e_i,\beta}(z-w)\widetilde{S}_i(z)\BKt_w(z)\intd\lambda_t(z),&\\
					\hspace{10cm}v(w)=0, \BFt_{n+|\alpha|}\phi(0)\leq\infty,&
				\end{cases}
				\nonumber
			\end{flalign}
			where
			\[
			\widetilde{S}_i(z,w)=\frac{(1-|z|^2)\partial_{z_i}\big[(1-\la z,w\ra)^{|\alpha|}v(z)\big]}{(1-\la w,z\ra)(1-\la z,w\ra)^{|\alpha|}}.
			\]
		\end{enumerate}	
	\end{lem}
	
	With Lemma \ref{lem: formula Bn}, we show the following.
	\begin{lem}\label{lem: formula F induction}
		Suppose $k$ is a non-negative integer and $\Gamma\subset\ind\times\ind$ is a finite set of multi-indices with $|\alpha|=|\beta|=k$ for every $(\alpha,\beta)\in\Gamma$. Suppose for some $\epsilon>-1-t$, $\{F_{\alpha,\beta}\}_{(\alpha,\beta)\in\Gamma}\subset\mathscr{C}^{2}(\bn\times\bn)$ and
		\begin{equation}\label{eqn: F assumption 1}
			\bigg|\sum_{(\alpha,\beta)\in\Gamma}I^{\alpha,\beta}(z-w)F_{\alpha,\beta}(z,w)\bigg|\lesssim|\varphi_z(w)|^{2k}|1-\la z,w\ra|^{2k+\epsilon},
		\end{equation}
		\begin{equation}\label{eqn: F assumption 2}
			\bigg|\sum_{j=1}^n\sum_{(\alpha,\beta)\in\Gamma}I^{\alpha,\beta+e_j}(z-w)\bpartial_{w_j} F_{\alpha,\beta}(z,w)\bigg|\lesssim|\varphi_z(w)|^{2k+1}|1-\la z,w\ra|^{2k+\epsilon}.
		\end{equation}
		
		Then
		\begin{flalign}\label{eqn: formula F induction}
			&\int_{\bn^2}\Phi^{(t)}_{n,k}(|\varphi_z(w)|^2)\frac{\sum_{(\alpha,\beta)\in\Gamma}I^{\alpha,\beta}(z-w)F_{\alpha,\beta}(z,w)}{|1-\la z,w\ra|^{2k}}\BKt_w(z)\intd\lambda_t(w)\intd\lambda_t(z)\nonumber\\
			=&\frac{\BFt_{n+k}\Phi^{(t)}_{n,k}(0)}{B(n,t+1)}\int_{\bn}(1-|z|^2)^{-2k}\sum_{(\alpha,\beta)\in\Gamma}d_{\alpha,\beta}(z)F_{\alpha,\beta}(z,z)\intd\lambda_t(z)\\
			&-\int_{\bn^2}\Phi^{(t)}_{n,k+1}(|\varphi_z(w)|^2)\frac{\sum_{i,j=1}^n\sum_{(\alpha,\beta)\in\Gamma}I^{\alpha+e_i,\beta+e_j}(z-w)D_{i,j}F_{\alpha,\beta}(z,w)}{|1-\la z,w\ra|^{2(k+1)}}\BKt_w(z)\intd\lambda_t(z)\intd\lambda_t(w).\nonumber
		\end{flalign}
		Here $D_{i,j}$ denotes the operation
		\[
		D_{i,j}=(1-\la z,w\ra)^2\partial_{z_i}\bpartial_{w_j}.
		\]
	\end{lem}
	Lemma \ref{lem: formula F induction} will be a key ingredient in the proof of Theorem \ref{thm: higher dimensions}. For the rest of this section, we prove the two lemmas. As in Section \ref{sec: trace formulas on the disk}, we start by proving a version of Lemma \ref{lem: formula Bn} at the point $0$.
	\begin{lem}\label{lem: formula Bn 0}
		Suppose $t>-1$, $k, l$ are non-negative integers with $k\geq l$, and $\Gamma\subset\ind\times\ind$ is a finite set of multi-indices with $|\kappa|=k, |\gamma|=l$ for every $(\kappa,\gamma)\in\Gamma$. Suppose $\{c_{\kappa,\gamma}\}_{(\kappa,\gamma)\in\Gamma}\subset\CC$, $\phi:(0,1)\to[0,\infty)$ is measurable and $v\in\mathscr{C}^1(\bn)$ satisfies that both
		\[
		\int_{\bn}\phi(|\zeta|^2)\bigg[\sum_{(\kappa,\gamma)\in\Gamma}c_{\kappa,\gamma}\zeta^\kappa\bar{\zeta}^\gamma\bigg]v(\zeta)\intd\lambda_t(\zeta)
		\]
		and
		\[
		\int_{\bn}\BGt_{l+n}\phi(|\zeta|^2)(1-|\zeta|^2)\bigg[\sum_{(\kappa,\gamma)\in\Gamma}c_{\kappa,\gamma} \zeta^\kappa\bar{\zeta}^\gamma\bigg]\bar{R}v(\zeta)\intd\lambda_t(\zeta)
		\]
		converge absolutely. Then
		\begin{flalign}\label{eqn: formula Bn 0}
			&\int_{\bn}\phi(|\zeta|^2)\bigg[\sum_{(\kappa,\gamma)\in\Gamma}c_{\kappa,\gamma}\zeta^\kappa\bar{\zeta}^\gamma\bigg]v(\zeta)\intd\lambda_t(\zeta)\\
			=&\begin{cases}
				cv(0)+\int_{\bn}\BGt_{l+n}\phi(|\zeta|^2)(1-|\zeta|^2)\bigg[\sum_{(\kappa,\gamma)\in\Gamma}c_{\kappa,\gamma} \zeta^\kappa\bar{\zeta}^\gamma\bigg]\bar{R}v(\zeta)\intd\lambda_t(\zeta),&v(0)\neq0, \BFt_{n+l}\phi(0)<\infty.\\
				\int_{\bn}\BGt_{l+n}\phi(|\zeta|^2)(1-|\zeta|^2)\bigg[\sum_{(\kappa,\gamma)\in\Gamma}c_{\kappa,\gamma} \zeta^\kappa\bar{\zeta}^\gamma\bigg]\bar{R}v(\zeta)\intd\lambda_t(\zeta),&v(0)=0, \BFt_{n+l}\phi(0)\leq\infty.
			\end{cases}\nonumber
		\end{flalign}
		Here
		\[
		c=\int_{\bn}\phi(|\zeta|^2)\bigg[\sum_{(\kappa,\gamma)\in\Gamma}c_{\kappa,\gamma}\zeta^\kappa\bar{\zeta}^\gamma\bigg]\intd\lambda_t(\zeta).
		\]
	\end{lem}
	
	\begin{proof}
		Assume first that $v(0)=0$.
		As in Appendix II we use $\phi_t$ to stand for the function $(1-s)^t$.
		By assumption on $v$, the first line of \eqref{eqn: formula Bn 0} converges absolutely. Therefore
		\begin{flalign*}
			&\int_{\bn}\phi(|\zeta|^2)\bigg[\sum_{(\kappa,\gamma)\in\Gamma}c_{\kappa,\gamma}\zeta^\kappa\bar{\zeta}^\gamma\bigg]v(\zeta)\intd\lambda_t(\zeta)\\
			=&\frac{(n-1)!}{\pi^nB(n,t+1)}\int_{\bn}\phi\phi_t(|\zeta|^2)\bigg[\sum_{(\kappa,\gamma)\in\Gamma}c_{\kappa,\gamma}\zeta^\kappa\bar{\zeta}^\gamma\bigg]v(\zeta)\intd m(\zeta)\\
			=&\frac{(n-1)!}{\pi^nB(n,t+1)}\int_0^1\phi\phi_t(r^2)\bigg[\sum_{(\kappa,\gamma)\in\Gamma}c_{\kappa,\gamma}\int_{r\sn}\zeta^\kappa\bar{\zeta}^\gamma v(\zeta)\intd\sigma_r(\zeta)\bigg]\intd r.
		\end{flalign*}

		Define $R=\sum_{i=1}^nz_i\partial_{z_i}$ be the radial derivative operator, and $\bar{R}=\sum_{i=1}^n\bar{z}_i\bpartial_{z_i}$. In Appendix I, Lemma \ref{lem: formula rSn 0}, we show that
		\begin{equation*}
			\int_{r\sn}\zeta^\kappa\bar{\zeta}^\gamma v(\zeta)\intd\sigma_r(\zeta)=2r^{2l+2n-1}\int_{r\bn}\frac{\zeta^\kappa\bar{\zeta}^\gamma}{|\zeta|^{2l+2n}}\bar{R}v(\zeta)\intd m(\zeta).
		\end{equation*}
		Plugging it back gives
		\begin{flalign*}
			&\int_{\bn}\phi(|\zeta|^2)\bigg[\sum_{(\kappa,\gamma)\in\Gamma}c_{\kappa,\gamma}\zeta^\kappa\bar{\zeta}^\gamma\bigg]v(\zeta)\intd\lambda_t(\zeta)\\
			=&\frac{(n-1)!}{\pi^nB(n,t+1)}\int_0^1\phi\phi_t(r^2)\bigg[\sum_{(\kappa,\gamma)\in\Gamma}c_{\kappa,\gamma}\cdot 2r^{2l+2n-1}\int_{r\bn}\frac{\zeta^\kappa\bar{\zeta}^\gamma}{|\zeta|^{2|\gamma|+2n}}\bar{R}v(\zeta)\intd m(\zeta)\bigg]\intd r\\
			=&\frac{(n-1)!}{\pi^nB(n,t+1)}\int_0^1\int_{r\bn}\phi\phi_t(r^2)2r^{2l+2n-1}\bigg[\sum_{(\kappa,\gamma)\in\Gamma}c_{\kappa,\gamma}\cdot \zeta^\kappa\bar{\zeta}^\gamma\bigg]|\zeta|^{-2l-2n}\bar{R}v(\zeta)\intd m(\zeta)\intd r\\
			=&\frac{(n-1)!}{\pi^nB(n,t+1)}\int_{\bn}\bigg[|\zeta|^{-2l-2n}\int_{|\zeta|}^1\phi\phi_t(r^2)2r^{2l+2n-1}\intd r\bigg]\bigg[\sum_{(\kappa,\gamma)\in\Gamma}c_{\kappa,\gamma} \zeta^\kappa\bar{\zeta}^\gamma\bigg]\bar{R}v(\zeta)\intd m(\zeta)\\
			=&\int_{\bn}\BGt_{l+n}\phi(|\zeta|^2)(1-|\zeta|^2)\bigg[\sum_{(\kappa,\gamma)\in\Gamma}c_{\kappa,\gamma} \zeta^\kappa\bar{\zeta}^\gamma\bigg]\bar{R}v(\zeta)\intd\lambda_t(\zeta).
		\end{flalign*}
		Here in the last equality we used that by Definition \ref{defn: BFt BGt},
		\[
		|\zeta|^{-2l-2n}\int_{|\zeta|}^1\phi\phi_t(r^2)2r^{2l+2n-1}\intd r\xlongequal{s=r^2}|\zeta|^{-2l-2n}\int_{|\zeta|^2}^1\phi\phi_t(s)s^{l+n-1}\intd s=(1-|\zeta|^2)^{t+1}\BGt_{l+n}\phi(|\zeta|^2).
		\]
		In the second to last equality, Fubini's theorem is applied: since by assumption the last integral converge absolutely, the condition for Fubini's theorem is satisfied. This proves the second case.
		
		Assuming that $v(0)\neq0$ and $\BFt_{l+n}\phi(0)<\infty$, then we have
		\begin{flalign*}
			&\int_{\bn}\bigg|\phi(|\zeta|^2)\bigg[\sum_{(\kappa,\gamma)\in\Gamma}c_{\kappa,\gamma}\zeta^\kappa\bar{\zeta}^\gamma\bigg]\big(v(\zeta)-v(0)\big)\bigg|\intd\lambda_t(\zeta)\\
			\leq&\int_{\bn}\bigg|\phi(|\zeta|^2)\bigg[\sum_{(\kappa,\gamma)\in\Gamma}c_{\kappa,\gamma}\zeta^\kappa\bar{\zeta}^\gamma\bigg]v(\zeta)\bigg|\intd\lambda_t(\zeta)+|v(0)|\int_{\bn}\bigg|\phi(|\zeta|^2)\bigg[\sum_{(\kappa,\gamma)\in\Gamma}c_{\kappa,\gamma}\zeta^\kappa\bar{\zeta}^\gamma\bigg]\bigg|\intd\lambda_t(\zeta).
		\end{flalign*}
		By assumption, the first integral converges. Also since $|\kappa|=k\geq l=|\gamma|$,
		\begin{flalign*}
			&\int_{\bn}\bigg|\phi(|\zeta|^2)\bigg[\sum_{(\kappa,\gamma)\in\Gamma}c_{\kappa,\gamma}\zeta^\kappa\bar{\zeta}^\gamma\bigg]\bigg|\intd\lambda_t(\zeta)\\
			\lesssim&\int_{\bn}\phi(|\zeta|^2)|\zeta|^{2l}\intd\lambda_t(\zeta)\\
			=&\frac{(n-1)!\sigma_{2n-1}}{\pi^nB(n,t+1)}\int_0^1\phi(r^2)r^{2n+2l-1}(1-r^2)^t\intd r\\
			\xlongequal{s=r^2}&\frac{(n-1)!\sigma_{2n-1}}{2\pi^nB(n,t+1)}\int_0^1\phi(s)s^{n+l-1}(1-s)^t\intd s\\
			=&\frac{(n-1)!\sigma_{2n-1}}{2\pi^nB(n,t+1)}\BFt_{n+l}\phi(0)\\
			<&\infty.
		\end{flalign*}
		Therefore we may apply the formula with $v(\zeta)-v(0)$ replacing $v(\zeta)$. This gives the first case.
	\end{proof}
	
	Suppose $\{e_1,\ldots,e_n\}$ and $\{f_1,\ldots,f_n\}$ are two orthonormal basis of $\cn$. Suppose
	\[
	\sum_{i=1}^n\zeta_ie_i=\sum_{i=1}^n\xi_if_i;\quad\sum_{i=1}^n\lambda_ie_i=\sum_{i=1}^n\omega_if_i.
	\]
	Then there is a unitary matrix $U=[u_{ij}]$ such that
	\[
	\xi=U\zeta,\quad \omega=U\lambda.
	\]
	Denote $U^*=[u_{ij}^*]$.
	Therefore
	\begin{equation}\label{eqn: unitary change}
		\sum_{i=1}^n\frac{\partial v}{\partial\bar{\zeta_i}}\bar{\lambda}_i=\sum_{i,j,k=1}^n\bigg(\frac{\partial v}{\partial\overline{\xi}_j}\bar{u}_{ji}\bar{\zeta}_i\bigg)\overline{\big(u^*_{ik}\omega_k\big)}=\sum_{j,k=1}^n\frac{\partial v}{\partial\bar{\xi}_j}\delta_{jk}\bar{w}_k=\sum_{j=1}^n\frac{\partial v}{\partial\bar{\xi}_j}\bar{w}_j.
	\end{equation}
	In other words, the function
	\[
	\sum_{i=1}^n\bpartial_iv\bar{\lambda}_i=\la\bpartial v,\lambda\ra
	\]
	does not depend on the choice of a basis.
	\begin{lem}\label{lem: inner product partial}
		Suppose $v\in\mathscr{C}^1(\bn)$ and $z\in\bn$. Then
		\begin{equation}\label{eqn: inner product partial}
			\la\bpartial_\zeta v(\zeta),\zeta\ra=-\frac{\la\bpartial \big(v\circ\varphi_z\big)(\varphi_z(\zeta)),z-\varphi_z(\zeta)\ra}{1-\la z,\zeta\ra}.
		\end{equation}
	\end{lem}
	
	\begin{proof}
		By \eqref{eqn: unitary change}, both sides of \eqref{eqn: inner product partial} do not depend on the choice of basis. Thus we may assume $z=(r,0,\ldots,0)$. In this case, we have the following expression
		\[
		\varphi_z(w)=\frac{1}{1-w_1r}\big(r-w_1,~-(1-r^2)^{1/2}w_2,~\ldots,~-(1-r^2)^{1/2}w_n\big).
		\]
		Consequently, we compute the Jacobian
		\[
		\bigg[\frac{\partial(\varphi_z)_i}{\partial w_j}\bigg]=\begin{bmatrix}
			-\frac{1-r^2}{(1-w_1r)^2}&0&0&0&\ldots&0\\
			-\frac{(1-r^2)^{1/2}w_2r}{(1-w_1r)^2}&-\frac{(1-r^2)^{1/2}}{1-w_1r}&0&0&\ldots&0\\
			-\frac{(1-r^2)^{1/2}w_3r}{(1-w_1r)^2}&0&-\frac{(1-r^2)^{1/2}}{1-w_1r}&0&\ldots&0\\
			\vdots&\vdots&\vdots&\vdots&\ddots&\vdots\\
			-\frac{(1-r^2)^{1/2}w_nr}{(1-w_1r)^2}&0&0&0&\ldots&-\frac{(1-r^2)^{1/2}}{1-w_1r}
		\end{bmatrix}.
		\]
		Therefore we have the following expression,
		\begin{equation}\label{eqn: 8}
			\la\bpartial\big(v\circ\varphi_z\big)(w),z-w\ra=\sum_{i,j=1}^n\bpartial_iv\big(\varphi_z(w)\big)\overline{\bigg(\frac{\partial\big(\varphi_z\big)_i}{\partial w_j}\bigg)}\overline{(z_j-w_j)}=\la\bpartial v\big(\varphi_z(w)\big),\xi\ra,
		\end{equation}
		where
		\[
		\xi_i=\sum_{j=1}^n\frac{\partial\big(\varphi_z\big)_i}{\partial w_j}(z_j-w_j)=\frac{\partial\big(\varphi_z\big)_i}{\partial w_1}(r-w_1)-\sum_{j=2}^n\frac{\partial\big(\varphi_z\big)_i}{\partial w_j}w_j.
		\]
		By the above, set
		\[
		\xi_1=-\frac{(1-r^2)(r-w_1)}{(1-w_1r)^2},
		\]
		and for $i=2,\ldots,n$,
		\[
		\xi_i=-\frac{(1-r^2)^{1/2}w_ir(r-w_1)}{(1-w_1r)^2}+\frac{(1-r^2)^{1/2}w_i}{1-w_1r}=\frac{(1-r^2)^{1/2}}{(1-w_1r)^2}\big(w_i-w_ir^2\big)=\frac{(1-r^2)^{3/2}}{(1-w_1r)^2}w_i.
		\]
		Thus we have
		\begin{equation}\label{eqn: 9}
			\xi=-\frac{1-r^2}{1-w_1r}\varphi_z(w).
		\end{equation}
		If we plug in $w=\varphi_z(\zeta)$ then by \eqref{eqn: 8}, \eqref{eqn: 9} and Lemma \ref{lem: Mobius basics}, we obtain the equalities
		\[
		\la\bpartial\big(v\circ\varphi_z\big)(\varphi_z(\zeta)),z-\varphi_z(\zeta)\ra=\la\bpartial v(\zeta),-\frac{1-|z|^2}{1-\la\varphi_z(\zeta),z\ra}\zeta\ra=-(1-\la z,\zeta\ra)\la\bpartial v(\zeta),\zeta\ra.
		\]
		Equivalently, \eqref{eqn: inner product partial} holds. This completes the proof of Lemma \ref{lem: inner product partial}.
	\end{proof}

	\begin{proof}[{\bf Proof of Lemma \ref{lem: formula Bn}}]
		By Lemma \ref{lem: Mobius basics} (4) (5),
		\begin{equation*}
			\BKt_w(z)\intd\lambda_t(w)\xlongequal[\zeta=\varphi_z(w)]{w=\varphi_z(\zeta)}\BKt_z(\zeta)\intd\lambda_t(\zeta),
		\end{equation*}
		and
		\begin{equation}\label{eqn: pf formula bn 1}
			z-\varphi_z(\zeta)=\frac{A_z\zeta}{1-\la \zeta,z\ra},
		\end{equation}
		where $A_zw$ is the linear transformation
		\[
		A_z\zeta=(1-|z|^2)P_z(\zeta)+(1-|z|^2)^{1/2}Q_z(\zeta).
		\]
		Write $|\alpha|=k, |\beta|=l$. Assume that $v(z)\neq 0$ and $\BFt_{n+l}\phi(0)<\infty$.
		\begin{flalign*}
			&\int_{\bn}\phi(|\varphi_z(w)|^2)I^{\alpha,\beta}(z-w)v(w)\BKt_w(z)\intd\lambda_t(w)\\
			\xlongequal{w=\varphi_z(\zeta)}&\int_{\bn}\phi(|\zeta|^2)I^{\alpha,\beta}(z-\varphi_z(\zeta))v\circ\varphi_z(\zeta)\BKt_z(\zeta)\intd\lambda_t(\zeta).
		\end{flalign*}
		By \eqref{eqn: pf formula bn 1},
		\begin{equation}\label{eqn: pf formula bn 2}
			I^{\alpha,\beta}(z-\varphi_z(\zeta))=\frac{\sum_{|\kappa|=k,|\gamma|=l}c_{\alpha,\beta,\kappa,\gamma,z}\zeta^\kappa\bar{\zeta}^\gamma}{(1-\la\zeta,z\ra)^k(1-\la z,\zeta\ra)^l}.
		\end{equation}
		By the above and Lemma \ref{lem: formula Bn 0},
		\begin{flalign*}
			&\int_{\bn}\phi(|\varphi_z(w)|^2)I^{\alpha,\beta}(z-w)v(w)\BKt_w(z)\intd\lambda_t(w)\\
			=&\int_{\bn}\phi(|\zeta|^2)\bigg[\sum_{|\kappa|=k,|\gamma|=l}c_{\alpha,\beta,\kappa,\gamma,z}\zeta^\kappa\bar{\zeta}^\gamma\bigg]\frac{v\circ\varphi_z(\zeta)}{(1-\la\zeta,z\ra)^k(1-\la z,\zeta\ra)^l}\BKt_z(\zeta)\intd\lambda_t(\zeta)\\
			=&cv(z)\\
			&\quad +\int_{\bn}\BGt_{l+n}\phi(|\zeta|^2)(1-|\zeta|^2)\bigg[\sum_{|\kappa|=k,|\gamma|=l}c_{\alpha,\beta,\kappa,\gamma,z}\zeta^\kappa\bar{\zeta}^\gamma\bigg]\bar{R}\bigg[\frac{v\circ\varphi_z(\zeta)}{(1-\la\zeta,z\ra)^k(1-\la z,\zeta\ra)^l}\BKt_z(\zeta)\bigg]\intd\lambda_t(\zeta)\\
			=&cv(z)\\
			&\ +\int_{\bn}\BGt_{l+n}\phi(|\zeta|^2)(1-|\zeta|^2)\bigg[\sum_{|\kappa|=k,|\gamma|=l}c_{\alpha,\beta,\kappa,\gamma,z}\zeta^\kappa\bar{\zeta}^\gamma\bigg]\bar{R}\bigg[\frac{v\circ\varphi_z(\zeta)}{(1-\la z,\zeta\ra)^l}\bigg](1-\la\zeta,z\ra)^{-k}\BKt_z(\zeta)\intd\lambda_t(\zeta)\\
			=&cv(z)+\int_{\bn}\BGt_{l+n}\phi(|\zeta|^2)(1-|\zeta|^2)I^{\alpha,\beta}(z-\varphi_z(\zeta))(1-\la z,\zeta\ra)^l\bar{R}\bigg[\frac{v\circ\varphi_z(\zeta)}{(1-\la z,\zeta\ra)^l}\bigg]\BKt_z(\zeta)\intd\lambda_t(\zeta).
		\end{flalign*}
		Write $h(\zeta)=\frac{v\circ\varphi_z(\zeta)}{(1-\la z,\zeta\ra)^l}$. Then
		\begin{equation*}
			h\circ\varphi_z(w)=\frac{(1-\la z,w\ra)^lv(w)}{(1-|z|^2)^l}.
		\end{equation*}
		By Lemma \ref{lem: inner product partial},
		\begin{equation*}
			\bar{R}\bigg[\frac{v\circ\varphi_z(\zeta)}{(1-\la z,\zeta\ra)^l}\bigg]=\bar{R}h(\zeta)=\la\bpartial_\zeta h(\zeta),\zeta\ra=-\frac{\la\bpartial\big(h\circ\varphi_z\big)(\varphi_z(\zeta)),z-\varphi_z(\zeta)\ra}{1-\la z,\zeta\ra}.
		\end{equation*}
		Taking the change of variable $\zeta=\varphi_z(w)$ we get
		\begin{flalign*}
			&\int_{\bn}\BGt_{l+n}\phi(|\zeta|^2)(1-|\zeta|^2)I^{\alpha,\beta}(z-\varphi_z(\zeta))(1-\la z,\zeta\ra)^l\bar{R}\bigg[\frac{v\circ\varphi_z(\zeta)}{(1-\la z,\zeta\ra)^l}\bigg]\BKt_z(\zeta)\intd\lambda_t(\zeta)\\
			=&-\int_{\bn}\BGt_{l+n}\phi(|\zeta|^2)(1-|\zeta|^2)\\
			&\qquad \qquad \qquad I^{\alpha,\beta}(z-\varphi_z(\zeta))(1-\la z,\zeta\ra)^{l-1}\la\bpartial\big(h\circ\varphi_z\big)(\varphi_z(\zeta)),z-\varphi_z(\zeta)\ra\BKt_z(\zeta)\intd\lambda_t(\zeta)\\
			=&-\int_{\bn}\BGt_{l+n}\phi(|\varphi_z(w)|^2)(1-|\varphi_z(w)|^2)\\
			&\qquad \qquad \qquad I^{\alpha,\beta}(z-w)\bigg[\frac{1-|z|^2}{1-\la z,w\ra}\bigg]^{l-1}\la\bpartial\bigg[\frac{(1-\la z,w\ra)^lv(w)}{(1-|z|^2)^l}\bigg](w),z-w\ra\BKt_w(z)\intd\lambda_t(w)\\
			=&-\sum_{j=1}^n\int_{\bn}\BGt_{l+n}\phi(|\varphi_z(w)|^2)I^{\alpha,\beta+e_j}(z-w)\frac{1-|w|^2}{1-\la w,z\ra}\frac{\bpartial_j\big[(1-\la z,w\ra)^lv(w)\big]}{(1-\la z,w\ra)^l}\BKt_w(z)\intd\lambda_t(w)\\
			=&-\sum_{j=1}^n\int_{\bn}\BGt_{l+n}\phi(|\varphi_z(w)|^2)I^{\alpha,\beta+e_j}(z-w)S_j(w)\BKt_w(z)\intd\lambda_t(w).
		\end{flalign*}
		To find the constant $c$, recall that by Lemma \ref{lem: formula Bn 0},
		\begin{equation*}
			c=\int_{\bn}\phi(|\zeta|^2)\bigg[\sum_{|\kappa|=l,|\gamma|=l}c_{\alpha,\beta,\kappa,\gamma,z}\zeta^\kappa\bar{\zeta}^\gamma\bigg]\intd\lambda_t(\zeta).
		\end{equation*}
		Clearly if $k\neq l$ then $c=0$. Assuming $k=l$, then by \eqref{eqn: pf formula bn 2}, we compute $c$ as follows, 
		\begin{flalign*}
			c=&\int_{\bn}\phi(|\zeta|^2)|1-\la z,\zeta\ra|^{2l}I^{\alpha,\beta}(z-\varphi_z(\zeta))\intd\lambda_t(\zeta)\\
			=&\int_{\bn}\phi(|\zeta|^2)I^{\alpha,\beta}(A_z\zeta)\intd\lambda_t(\zeta)\\
			=&\frac{(n-1)!}{\pi^nB(n,t+1)}\int_0^1\phi(r^2)r^{2n++2l-1}(1-r^2)^t\bigg[\int_{\sn}I^{\alpha,\beta}(A_z\zeta)\intd\sigma(\zeta)\bigg]\intd r\\
			=&\frac{(n-1)!}{\pi^nB(n,t+1)}\cdot\frac{1}{2}\BFt_{n+l}\phi(0)\cdot\sigma_{2n-1}d_{\alpha,\beta}(z)\\
			=&\frac{\BFt_{n+l}\phi(0)}{B(n,t+1)}d_{\alpha,\beta}(z).
		\end{flalign*}
		This proves the first case of \eqref{eqn: formula Bn d''}. The second case is proved in the same way. 
		
		In \eqref{eqn: formula Bn d''}, reverse $(z, w)$, $(\alpha,\beta)$, and replace $v$ with $\bar{v}$. Then
		\begin{flalign*}
			&\int_{\bn}\phi(|\varphi_w(z)|^2)I^{\beta,\alpha}(w-z) \bar{v}(z)\BKt_z(w)\intd\lambda_t(z)\\
			=&\begin{cases}
				\frac{d_{\beta,\alpha}(w)}{B(n,t+1)}\cdot \BFt_{n+|\alpha|}\phi(0) \bar{v}(w)-\sum_{j=1}^n\int_{\bn}\BGt_{|\alpha|+n}\phi(|\varphi_w(z)|^2)I^{\beta,\alpha+e_j}(w-z)S_j(z)\BKt_z(w)\intd\lambda_t(z),&\\
				\hspace{10cm} v(w)\neq0, \BFt_{n+|\alpha|}\phi(0)<\infty,&\\
				&\\
				-\sum_{j=1}^n\int_{\bn}\BGt_{|\alpha|+n}\phi(|\varphi_w(z)|^2)I^{\beta,\alpha+e_j}(w-z)S_j(z)\BKt_z(w)\intd\lambda_t(z),&\\
				\hspace{10cm}v(w)=0, \BFt_{n+|\alpha|}\phi(0)\leq\infty,&
			\end{cases}
			\nonumber
		\end{flalign*}
		where
		\[
		S_j(z)=\frac{(1-|z|^2)\bpartial_{z_j}\big[(1-\la w,z\ra)^{|\alpha|}\bar{v}(z)\big]}{(1-\la z,w\ra)(1-\la w,z\ra)^{|\alpha|}}.
		\]
		Taking conjugate on both sides, we get Equation (\eqref{eqn: formula Bn d'}) from
		\[
		\overline{d_{\alpha,\beta}(z)}=d_{\beta,\alpha}(z),\quad|\varphi_z(w)|=|\varphi_w(z)|.
		\]
		This completes the proof.
	\end{proof}

	\begin{proof}[{\bf Proof of Lemma \ref{lem: formula F induction}}]
		By Estimates \eqref{eqn: Phi behavior}, \eqref{eqn: Rudin-Forelli 3}, and assumption \eqref{eqn: F assumption 1}, we conclude that the left hand side of Equation \eqref{eqn: formula F induction} is absolutely integrable. For each $z\in\bn, (\alpha,\beta)\in\Gamma$, we compute the following integral, 
		\begin{flalign*}
			&\int_{\bn}\Phi^{(t)}_{n,k}(|\varphi_z(w)|^2)I^{\alpha,\beta}(z-w)\frac{F_{\alpha,\beta}(z,w)}{|1-\la z,w\ra|^{2k}}\BKt_w(z)\intd\lambda_t(w)\\
			\xlongequal{\eqref{eqn: formula Bn d''}}&\frac{d_{\alpha,\beta}(z)}{B(n,t+1)}\BFt_{n+k}\Phi^{(t)}_{n,k}(0)\frac{F_{\alpha,\beta}(z,z)}{(1-|z|^2)^{2k}}\\
			&-\sum_{j=1}^n\int_{\bn}\BGt_{n+k}\Phi^{(t)}_{n,k}(|\varphi_z(w)|^2)I^{\alpha,\beta+e_j}(z-w)\frac{(1-|w|^2)\bpartial_{w_j}F_{\alpha,\beta}(z,w)}{(1-\la w,z\ra)^{k+1}(1-\la z,w\ra)^k}\BKt_w(z)\intd\lambda_t(w).
		\end{flalign*}
		By \eqref{eqn: GPhi behavior}, \eqref{eqn: Rudin-Forelli 3} and assumption \eqref{eqn: F assumption 2}, the integral
		\[
		\int_{\bn^2}\BGt_{n+k}\Phi^{(t)}_{n,k}(|\varphi_z(w)|^2)\frac{\sum_{j=1}^n\sum_{(\alpha,\beta)\in\Gamma}I^{\alpha,\beta+e_j}(z-w)(1-|w|^2)\bpartial_{w_j}F(z,w)}{(1-\la w,z\ra)^{k+1}(1-\la z,w\ra)^k}\BKt_w(z)\intd\lambda_t(w)\intd\lambda_t(z)
		\]
		converges absolutely.
		Therefore the first line of \eqref{eqn: formula F induction} equals
		\begin{flalign*}
			&\int_{\bn}\sum_{(\alpha,\beta)\in\Gamma}\frac{d_{\alpha,\beta}(z)}{B(n,t+1)}\BFt_{n+k}\Phi^{(t)}_{n,k}(0)\frac{F_{\alpha,\beta}(z,z)}{(1-|z|^2)^{2k}}\intd\lambda_t(z)\\
			&\qquad \qquad -\int_{\bn^2}\BGt_{n+k}\Phi^{(t)}_{n,k}(|\varphi_z(w)|^2)\\
			&\qquad \qquad \qquad \frac{\sum_{j=1}^n\sum_{(\alpha,\beta)\in\Gamma}I^{\alpha,\beta+e_j}(z-w)(1-|w|^2)\bpartial_{w_j}F(z,w)}{(1-\la w,z\ra)^{k+1}(1-\la z,w\ra)^k}\BKt_w(z)\intd\lambda_t(z)\intd\lambda_t(w).
		\end{flalign*}
		Again, for any $w\in\bn$, $j=1,\ldots,n$ and $(\alpha,\beta)\in\Gamma$, we compute the following integral,
		\begin{flalign*}
			&\int_{\bn}\BGt_{n+k}\Phi^{(t)}_{n,k}(|\varphi_z(w)|^2)I^{\alpha,\beta+e_j}(z-w)\frac{(1-|w|^2)\bpartial_{w_j}F(z,w)}{(1-\la w,z\ra)^{k+1}(1-\la z,w\ra)^k}\BKt_w(z)\intd\lambda_t(z)\\
			\xlongequal{\eqref{eqn: formula Bn d'}}&\sum_{i=1}^n\int_{\bn}\big(\BGt_{n+k}\big)^2\Phi^{(t)}_{n,k}(|\varphi_z(w)|^2)\\
			&\qquad \qquad \qquad I^{\alpha+e_i,\beta+e_j}(z-w)\frac{(1-|z|^2)(1-|w|^2)\partial_{z_i}\bpartial_{w_j}F(z,w)}{(1-\la w,z\ra)^{k+2}(1-\la z,w\ra)^k}\BKt_w(z)\intd\lambda_t(z)\\
			=&\sum_{i=1}^n\int_{\bn}M_{\phi_1}\big(\BGt_{n+k}\big)^2\Phi^{(t)}_{n,k}(|\varphi_z(w)|^2)\\
			&\qquad \qquad \qquad I^{\alpha+e_i,\beta+e_j}(z-w) \frac{\partial_{z_i}\bpartial_{w_j}F(z,w)}{(1-\la w,z\ra)^{k+1}(1-\la z,w\ra)^{k-1}}\BKt_w(z)\intd\lambda_t(z)\\
			=&\sum_{i=1}^n\int_{\bn}\Phi^{(t)}_{n,k+1}(|\varphi_z(w)|^2)I^{\alpha+e_i,\beta+e_j}(z-w)\frac{D_{i,j}F(z,w)}{|1-\la w,z\ra|^{2(k+1)}}\BKt_w(z)\intd\lambda_t(z).
		\end{flalign*}
		Altogether, the first line of \eqref{eqn: formula F induction} equals
		\begin{flalign*}
			&\int_{\bn}\sum_{(\alpha,\beta)\in\Gamma}\frac{d_{\alpha,\beta}(z)}{B(n,t+1)}\BFt_{n+k}\Phi^{(t)}_{n,k}(0)\frac{F_{\alpha,\beta}(z,z)}{(1-|z|^2)^{2k}}\intd\lambda_t(z)\\
			&+\int_{\bn^2}\Phi^{(t)}_{n,k+1}(|\varphi_z(w)|^2)\frac{\sum_{i=1,j}^n\sum_{(\alpha,\beta)\in\Gamma}I^{\alpha+e_i,\beta+e_j}(z-w)D_{i,j}F(z,w)}{|1-\la w,z\ra|^{2(k+1)}}\BKt_w(z)\intd\lambda_t(z)\intd\lambda_t(w).
		\end{flalign*}
		This completes the proof of Lemma \ref{lem: formula F induction}.
	\end{proof}
	
	Using the same proof of Lemma \ref{lem: formula Bn}, one can show the following.
	\begin{lem}\label{lem: BM hardy}
		Suppose $\alpha, \beta\in\ind$, and $v\in\mathscr{C}^1(\overline{\bn})$. Then the following hold.
		\begin{enumerate}
			\item If $|\alpha|\geq|\beta|$, then
			\begin{flalign}\label{eqn: BM hardy d''}
				&\int_{\sn}I^{\alpha,\beta}(z-w)v(w)K_w(z)\frac{\intd\sigma(w)}{\sigma_{2n-1}}\\
				=&d_{\alpha,\beta}(z)v(z)-\frac{1}{n}\sum_{j=1}^n\int_{\bn}|\varphi_z(w)|^{-2|\beta|-2n}I^{\alpha,\beta+e_j}(z-w)\frac{\bpartial_j\big[(1-\la z,w\ra)^{|\beta|}v(w)\big]}{(1-\la z,w\ra)^{|\beta|}(1-\la w,z\ra)}K_w(z)\intd\lambda_0(w),\nonumber
			\end{flalign}
			\item If $|\alpha|\leq|\beta|$, then
			\begin{flalign}\label{eqn: BM hardy d'}
				&\int_{\sn}I^{\alpha,\beta}(z-w)v(z)K_w(z)\frac{\intd\sigma(z)}{\sigma_{2n-1}}\\
				=&d_{\beta,\alpha}(w)v(w)+\frac{1}{n}\sum_{i=1}^n\int_{\bn}|\varphi_z(w)|^{-2|\alpha|-2n}I^{\alpha+e_i,\beta}(z-w)\frac{\partial_i\big[(1-\la z,w\ra)^{|\alpha|}v(z)\big]}{(1-\la z,w\ra)^{|\alpha|}(1-\la w,z\ra)}K_w(z)\intd\lambda_0(z),\nonumber
			\end{flalign}
		\end{enumerate}
	\end{lem}

	\section{The Higher Dimensions}\label{sec: the higher dimensions}
	The goal of this section is to prove Theorem \ref{thm: higher dimensions}. To start with, we apply Lemma \ref{lem: formula F induction} to get the following.
	
	\begin{lem}\label{lem: semicom = PR}
		Suppose $t>-1$ and $f, g\in\mathscr{C}^1(\bn)$. Suppose $f, g, \partial f, \bpartial g$ are bounded on $\bn$. Then
		\begin{equation*}
			\BTt_f\BTt_g-\BTt_{fg}=\BPt R,
		\end{equation*}
		where $R:\bert\to L^2(\lambda_t)$ is defined by
		\[
		Rh(z)=-\int_{\bn}\Phi^{(t)}_{n,1}(|\varphi_z(w)|^2)\frac{1-\la z,w\ra}{1-\la w,z\ra}\la\partial_zf,\overline{z-w}\ra\la\bpartial_wg,z-w\ra h(w)\BKt_w(z)\intd\lambda_t(w).
		\]
	\end{lem}
	
	\begin{proof}
		By definition, for $h\in H^\infty(\bn)$,
		\begin{equation*}
			\bigg(\BTt_f\BTt_g-\BTt_{fg}\bigg)h(\xi)=\int_{\bn}\int_{\bn}\big(f(z)g(w)-f(w)g(w)\big)h(w)\BKt_w(z)\BKt_z(\xi)\intd\lambda_t(w)\intd\lambda_t(z).
		\end{equation*}
		Denote $F_\xi(z,w)=\big(f(z)g(w)-f(w)g(w)\big)h(w)\BKt_z(\xi)$. Then
		\[
		F(z,z)=0.
		\]
		For fixed $\xi\in\bn$, $F_\xi(z,w)$ is bounded, and by Lemma \ref{lem: Mobius basics},
		\begin{equation*}
			\bigg|\sum_{j=1}^nI^{0,e_j}(z-w)\bpartial_{w_j}F(z,w)\bigg|\lesssim|z-w|\lesssim|\varphi_z(w)||1-\la z,w\ra|^{1/2}.
		\end{equation*}
		Then the assumption of Lemma \ref{lem: formula F induction} is satisfied when we take $\Gamma=\{(0,0)\}$, $k=0$, $\epsilon=0$ and $F_{0,0}=F_\xi$. Applying the lemma, we obtain the following computation,
		\begin{flalign*}
			&\bigg(\BTt_f\BTt_g-\BTt_{fg}\bigg)h(\xi)\\
			=&\int_{\bn\times\bn}\Phi^{(t)}_{n,0}(|\varphi_z(w)|^2)F(z,w)\BKt_w(z)\intd\lambda_t(w)\intd\lambda_t(z)\\
			=&-\int_{\bn\times\bn}\Phi^{(t)}_{n,1}(|\varphi_z(w)|^2)\frac{\sum_{i,j=1}^nI^{e_i,e_j}(z-w)D_{i,j}F(z,w)}{|1-\la z,w\ra|^2}\BKt_w(z)\intd\lambda_t(z)\intd\lambda_t(w)\\
			=&-\int_{\bn\times\bn}\Phi^{(t)}_{n,1}(|\varphi_z(w)|^2)\frac{1-\la z,w\ra}{1-\la w,z\ra}\\
			&\qquad \qquad \qquad \sum_{i,j=1}^nI^{e_i,e_j}(z-w)\partial_if(z)\bpartial_jg(w)h(w)\BKt_z(\xi)\BKt_w(z)\intd\lambda_t(z)\intd\lambda_t(w)\\
			=&\BPt Rh(\xi).
		\end{flalign*}
		This completes the proof.
	\end{proof}

	\begin{lem}\label{lem: semicom in trace class n>1}
		Suppose $t>2n-3$ and $f, g$ satisfy Condition 1. Then the semicommutator $\BTt_f\BTt_g-\BTt_{fg}$ belongs to the trace class.
	\end{lem}
	
	\begin{proof}
		Divide $\BTt_f\BTt_g-\BTt_{fg}=\BPt R$ as in Lemma \ref{lem: semicom = PR}.
		Take $\epsilon>0$ so that $t>2n-3+2\epsilon$. Let $c=n+\epsilon$ and denote $\hat{R}:L_{a,t+2c}^2(\bn)\to L^2(\lambda_t)$ the integral operator with the same integral formula as $R$, i.e.,
		\[
		\hat{R}h(z)=-\int_{\bn}\Phi^{(t)}_{n,1}(|\varphi_z(w)|^2)\frac{1-\la z,w\ra}{1-\la w,z\ra}\la\partial_zf,\overline{z-w}\ra\la\bpartial_wg,z-w\ra h(w)\BKt_w(z)\intd\lambda_t(w).
		\]
		Let $E:\bert\to L_{a,t+2c}^2(\bn)$ be the embedding map. Split $\BTt_f\BTt_g-\BTt_{fg}$ as 
		\[
		\BTt_f\BTt_g-\BTt_{fg}:\bert\xlongrightarrow{E}L_{a,t+2c}^2(\bn)\xlongrightarrow{\hat{R}}L^2(\lambda_t)\xlongrightarrow{\BPt}\bert.
		\]
		It is well-known that $E$ is in the trace class \cite{Ta-Wa-Zh:HHtrace}. It remains to show that $\hat{R}$ is bounded. By definition, the operator $\hat{R}$ has integral kernel
		\[
		\hat{F}(z,w)=C\Phi^{(t)}_{n,1}(|\varphi_z(w)|^2)\frac{1-\la z,w\ra}{1-\la w,z\ra}\la\partial_zf,\overline{z-w}\ra\la\bpartial_wg,z-w\ra\BKt_w(z)(1-|w|^2)^{-2c},
		\]
		where $C$ is a constant. By assumption,
		\[
		|\hat{F}(z,w)|\lesssim\frac{\Phi^{(t)}_{n,1}(|\varphi_z(w)|^2)|\varphi_z(w)|^2}{|1-\la z,w\ra|^{1+t-\epsilon}(1-|w|^2)^{2c}}.
		\]
		By Lemma \ref{lem: Phi asymptotic behavior},
		\[
		\Phi^{(t)}_{n,1}(s)\lesssim s^{-n-1/2}(1-s).
		\]
		Take $x$ so that
		\[
		-2(n+\epsilon)+1>x>-2-t.
		\]
		Let $y=n+\epsilon+x$, $p(w)=(1-|w|^2)^x$ and $ q(z)=(1-|z|^2)^y$. Then by Lemma \ref{lem: Rudin Forelli generalizations} (3),
		\[
		\int_{\bn}|\hat{F}(z,w)|p(w)\intd\lambda_{t+2c}(w)\lesssim\int_{\bn}\Phi_{n,1}^{(t)}(|\varphi_z(w)|^2)|\varphi_z(w)|^2\frac{(1-|w|^2)^{t+x}}{|1-\la z,w\ra|^{1+t-\epsilon}}\intd m(w)\lesssim q(z),
		\]
		\[
		\int_{\bn}|\hat{F}(z,w)|q(z)\intd\lambda_t(z)\lesssim(1-|w|^2)^{-2c}\int_{\bn}\Phi^{(t)}_{n,1}(|\varphi_z(w)|^2)|\varphi_z(w)|^2\frac{(1-|z|^2)^{t+y}}{|1-\la z,w\ra|^{1+t-\epsilon}}\intd m(z)\lesssim p(w).
		\]
		By Schur's test, $\hat{R}$ is bounded. Therefore the semicommutator is in the trace class. This completes the proof.
	\end{proof}
	
	Recall that the operations $\BFt_m, \BGt_m$ and the functions $\Phi^{(t)}_{n,k}$ are defined in Appendix II.

	\begin{lem}\label{lem: FPhi(0) expression}
		We have
		\[
		\BFt_{n+1}\Phi^{(t)}_{n,1}(0)=\sum_{j=0}^\infty \frac{B(n+1+j,t+1)}{1+j}=-\int_0^1(1-s)^{n-1}s^t\ln s\intd s.
		\]
		Consequently, as $t$ tends to infinity,
		\[
		\BFt_{n+1}\Phi^{(t)}_{n,1}(0)=n!t^{-n-1}+o(t^{-n-2}).
		\]
	\end{lem}
	
	\begin{proof}
		First, by Lemma \ref{lem: FG properties},
		\[
		\begin{split}
		&\BFt_{n+1}\Phi^{(t)}_{n,1}(0)=\BFt_{n+1}M_{\phi_1}\big(\BGt_n\big)^21(0)=\BFt_{n+1}\BGt_n1(0)\\
		=&\sum_{j=0}^\infty\frac{1}{1+j}\BFt_{n+1+j}1(0)=\sum_{j=0}^\infty\frac{1}{1+j}B(n+1+j,t+1).
		\end{split}
		\]
		By definition, the above equals
		\begin{equation*}
			\sum_{j=0}^\infty\frac{1}{1+j}\int_0^1(1-s)^{n+j}s^t\intd s=-\int_0^1(1-s)^{n-1}s^t\ln s\intd s.
		\end{equation*}
		This proves the first line of equations. The second line of equation follows from the estimate
		\[
		|-\ln s-1+s|\lesssim(1-s)^2s^{-1}.
		\]
		This completes the proof.
	\end{proof}
	
	\begin{lem}\label{lem: Phi n 2}
		We have
		\[
		\Phi^{(t)}_{n,2}(s)=(1-s)^{-t}s^{-n-1}\sum_{k=1}^n\frac{(n-1)!\Gamma(t+1)}{(n-k)!\Gamma(t+1+k)}\int_s^1F(s,x)x^{n-k-1}(1-x)^{t+k-1}\intd x.
		\]
	\end{lem}
	
	\begin{proof}
		By Definition \ref{defn: BFt BGt}, we compute $\Phi^{(t)}_{n,2}(s)$ as follows,
		\begin{flalign*}
			&\Phi^{(t)}_{n,2}(s)\\
			=&(1-s)\big(\BGt_{n+1}\big)^2M_{\phi_1}\big(\BGt_n\big)^21(s)\\
			=&(1-s)^{-t}s^{-n-1}\int_s^1s_1^n(1-s_1)^t\BGt_{n+1}M_{\phi_1}\big(\BGt_n\big)^21(s_1)\intd s_1\\
			=&(1-s)^{-t}s^{-n-1}\int_s^1s_1^{-1}(1-s_1)^{-1}\int_{s_1}^1s_2^n(1-s_2)^{t+1}\big(\BGt_n\big)^21(s_2)\intd s_2\intd s_1\\
			=&(1-s)^{-t}s^{-n-1}\int_s^1s_1^{-1}(1-s_1)^{-1}\int_{s_1}^1\int_{s_2}^1s_3^{n-1}(1-s_3)^t\BGt_n1(s_3)\intd s_3\intd s_2\intd s_1\\
			=&(1-s)^{-t}s^{-n-1}\int_s^1s_1^{-1}(1-s_1)^{-1}\int_{s_1}^1\int_{s_2}^1s_3^{-1}(1-s_3)^{-1}\int_{s_3}^1s_4^{n-1}(1-s_4)^t\intd s_4\intd s_3\intd s_2\intd s_1\\
			=&(1-s)^{-t}s^{-n-1}\iiiint_{s<s_1<s_2<s_3<s_4<1}s_1^{-1}(1-s_1)^{-1}s_3^{-1}(1-s_3)^{-1}s_4^{n-1}(1-s_4)^t\intd s_4\intd s_3\intd s_2\intd s_1\\
			=&(1-s)^{-t}s^{-n-1}\int_s^1\bigg\{\iint_{s<s_1<s_2<s_3}s_1^{-1}(1-s_1)^{-1}\intd s_1\intd s_2\bigg\}\\
			&\qquad \qquad \qquad \qquad \qquad \cdot\bigg\{\int_{s_3}^1s_4^{n-1}(1-s_4)^t\intd s_4\bigg\}s_3^{-1}(1-s_3)^{-1}\intd s_3.
		\end{flalign*}
		By Lemma \ref{lem: F(s,x) as integral}, we have the following integral, 
		\[
		\iint_{s<s_1<s_2<s_3}s_1^{-1}(1-s_1)^{-1}\intd s_1\intd s_2=F(s,s_3).
		\]
		For a positive integer $m$, and $x>-1$, temporarily denote $I(m,x)=\int_{s_3}^1s_4^{m-1}(1-s_4)^x\intd s_4$. Then, we obtain the following relations,
		\[
		I(1,x)=\frac{(1-s_3)^{x+1}}{x+1},
		\]
		and
		\begin{flalign*}
			I(m+1,x)=&\int_{s_3}^1s_4^{m}(1-s_4)^x\intd s_4\\
			=&-(x+1)^{-1}\int_{s_3}^1s_4^{m}\intd (1-s_4)^{x+1}\\
			=&\frac{1}{x+1}s_3^{m}(1-s_3)^{x+1}+\frac{m}{x+1}\int_{s_3}^1s_4^{m-1}(1-s_4)^{x+1}\intd s_4\\
			=&\frac{1}{x+1}s_3^{m}(1-s_3)^{x+1}+\frac{m}{x+1}I(m,x+1).
		\end{flalign*}
		Thus by induction, we obtain the following formula for $I(n,t)$, 
		\begin{equation*}
			\int_{s_3}^1s_4^{n-1}(1-s_4)^t\intd s_4=I(n,t)=\sum_{k=1}^n\frac{(n-1)!\Gamma(t+1)}{(n-k)!\Gamma(t+1+k)}s_3^{n-k}(1-s_3)^{t+k}.
		\end{equation*}
		Therefore, we conclude with the following formula for $\Phi^{(t)}_{n,2}$, 
		\begin{equation*}
			\Phi^{(t)}_{n,2}(s)
			=(1-s)^{-t}s^{-n-1}\sum_{k=1}^n\frac{(n-1)!\Gamma(t+1)}{(n-k)!\Gamma(t+1+k)}\int_s^1F(s,s_3)s_3^{n-k-1}(1-s_3)^{t+k-1}\intd s_3.
		\end{equation*}
		This completes the proof.
	\end{proof}
	
	The following lemma helps us study the first term of $\Tr\bigg(\BTt_f\BTt_g-\BTt_{fg}\bigg)$ after iteration.
	\begin{lem}\label{lem: d f g= f g log = f g R}
		For $f, g\in\mathscr{C}^1(\bn)$, the following are equal whenever the integrals converge.
		\begin{flalign}\label{eqn: d f g= f g log = f g R}
			\int_{\bn}\frac{\sum_{i,j=1}^nd_{e_i,e_j}(w)\partial_if(w)\bpartial_jg(w)}{(1-|w|^2)^{n+1}}\intd m(w)
			=&\frac{-1}{(2i)^nn!}\int_{\bn}\partial f\wedge\bpartial g\wedge\bigg[\partial\bpartial\log(1-|w|^2)\bigg]^{n-1}\\
			=&\frac{1}{n}\int_{\bn}\frac{\sum_{i=1}^n\partial_if(w)\bpartial_ig(w)-Rf(w)\bar{R}g(w)}{(1-|w|^2)^n}\intd m(w).\nonumber
		\end{flalign}
	\end{lem}
	
	\begin{proof}
		By Definition \ref{defn: d numbers I function}, we consider 
		\[
		\sum_{i,j=1}^nd_{e_i,e_j}(w)\partial_if(w)\bpartial_jg(w)=\int_{\sn}\bigg[\sum_{i=1}^n\big(A_w\zeta\big)_i\partial_if(w)\bigg]\bigg[\sum_{j=1}^n\overline{\big(A_w\zeta\big)_j}\bpartial_jg(w)\bigg]\frac{\intd\sigma(\zeta)}{\sigma_{2n-1}}.
		\]
		By an argument similar as in the proof of\eqref{eqn: unitary change}, the sum in each big bracket is independent of the choice of an orthonormal basis of $\mathbb{C}^n$. Thus the integrand in the left hand side of \eqref{eqn: d f g= f g log = f g R} does not depend on the choice of a basis. At each $w\in\bn$, $w\neq0$, choose a basis under which $w=(w_1,0,\ldots,0)$. By \eqref{eqn: d at (1,0...0)},
		\begin{equation}\label{eqn: proof d f g sum}
			\frac{\sum_{i,j=1}^nd_{e_i,e_j}(w)\partial_if(w)\bpartial_jg(w)}{(1-|w|^2)^{n+1}}=\frac{1}{n}\bigg[\frac{\partial_1f(w)\bpartial_1g(w)}{(1-|w|^2)^{n-1}}+\frac{\sum_{i=2}^n\partial_if(w)\bpartial_jg(w)}{(1-|w|^2)^n}\bigg].
		\end{equation}
		On the other hand, we compute
		\[
		\partial\bpartial\log(1-|w|^2)=-\partial\bigg[\frac{\sum_{j=1}^nw_j\intd \bar{w}_j}{1-|w|^2}\bigg]=-\frac{\sum_{i,j=1}^n\bar{w}_iw_j\intd w_i\wedge\intd\bar{w}_j}{(1-|w|^2)^2}-\frac{\sum_{j=1}^n\intd w_j\wedge\intd\bar{w}_j}{1-|w|^2}.
		\]
		At $w=(w_1,0,\ldots,0)$ the above equals
		\[
		-\frac{|w|^2\intd w_1\wedge\intd\bar{w}_1}{(1-|w|^2)^2}-\sum_{j=1}^n\frac{\intd w_j\wedge\intd\bar{w}_j}{1-|w|^2}=-\frac{\intd w_1\wedge\intd\bar{w}_1}{(1-|w|^2)^2}-\sum_{j=2}^n\frac{\intd w_j\wedge\intd\bar{w}_j}{1-|w|^2}.
		\]
		Thus, we have
		\[
		\bigg[\partial\bpartial\log(1-|w|^2)\bigg]^{n-1}\bigg|_{w=(w_1,0,\ldots,0)}=(-1)^{n-1}(n-1)!\bigg[\frac{\bigwedge_{j=2}^n\big(\intd w_j\wedge\intd\bar{w}_j\big)}{(1-|w|^2)^{n-1}}+\sum_{i=2}^n\frac{\bigwedge_{j\neq i}\big(\intd w_j\wedge\intd\bar{w}_j\big)}{(1-|w|^2)^n}\bigg].
		\]
		Therefore at $w=(w_1,0,\ldots,0)$, we have 
		\begin{flalign}\label{eqn: proof f g log wedge}
			&\partial f\wedge\bpartial g\wedge\bigg[\partial\bpartial\log(1-|w|^2)\bigg]^{n-1}\nonumber\\
			=&(-1)^{n-1}(n-1)!\bigg[\frac{\partial_1f(w)\bpartial_1g(w)}{(1-|w|^2)^{n-1}}+\frac{\sum_{i=2}^n\partial_if(w)\bpartial_ig(w)}{(1-|w|^2)^n}\bigg]\bigwedge_{j=1}^n\big(\intd w_j\wedge\intd\bar{w}_j\big)\nonumber\\
			=&-(2i)^n(n-1)!\bigg[\frac{\partial_1f(w)\bpartial_1g(w)}{(1-|w|^2)^{n-1}}+\frac{\sum_{i=2}^n\partial_if(w)\bpartial_ig(w)}{(1-|w|^2)^n}\bigg]\intd m(w).
		\end{flalign}
		Comparing \eqref{eqn: proof d f g sum} and \eqref{eqn: proof f g log wedge}, we conclude that at $w=(w_1,0,\ldots,0)$,
		\begin{equation*}
			\partial f\wedge\bpartial g\wedge\bigg[\partial\bpartial\log(1-|w|^2)\bigg]^{n-1}=-(2i)^nn!\frac{\sum_{i,j=1}^nd_{e_i,e_j}(w)\partial_if(w)\bpartial_jg(w)}{(1-|w|^2)^{n+1}}\intd m(w).
		\end{equation*}
		Since both sides are independent of the choice of basis, the equation holds for general $w$. This proves the first equality. 
		
		Also, it is easy to see that $\sum_{i=1}^n\partial_if(w)\bpartial_ig(w)-Rf(w)\bar{R}g(w)$ is invariant of the choice of a basis. Again, if one chooses a basis so that $w=(w_1,0,\ldots,0)$, then
		\begin{equation*}
			\sum_{i=1}^n\partial_if(w)\bpartial_ig(w)-Rf(w)\bar{R}g(w)=(1-|w|^2)\partial_1f(w)\bpartial_1g(w)+\sum_{i=2}^n\partial_if(w)\bpartial_ig(w).
		\end{equation*}
		Comparing the above and \eqref{eqn: proof d f g sum} gives
		\begin{equation}\label{eqn: proof sum d f g in R form}
			\sum_{i=1}^n\partial_if(w)\bpartial_ig(w)-Rf(w)\bar{R}g(w)=n\frac{\sum_{i,j=1}^nd_{e_i,e_j}(w)\partial_if(w)\bpartial_jg(w)}{1-|w|^2}.
		\end{equation}
		Since both sides are independent of the choice of a basis the equation holds for general $w$. Plugging \eqref{eqn: proof sum d f g in R form} into the first equality gives the second equality. This completes the proof.
	\end{proof}

	\begin{proof}[{\bf Proof of Theorem \ref{thm: higher dimensions}}]
		The fact that $\BTt_f\BTt_g-\BTt_{fg}$ belongs to the trace class is proved in Lemma \ref{lem: semicom in trace class n>1}. By Lemma \ref{lem: trace is integral of Berezin} and Lemma \ref{lem: semicom = PR}, we compute the trace of the semi-commutator as follows, 
		\begin{flalign*}
			&\Tr\bigg(\BTt_f\BTt_g-\BTt_{fg}\bigg)\\
			=&\int_{\bn}\la\bigg(\BTt_f\BTt_g-\BTt_{fg}\bigg)\BKt_\xi,\BKt_\xi\ra\\
			=&-\int_{\bn}\bigg\{\int_{\bn}\int_{\bn}\Phi^{(t)}_{n,1}(|\varphi_z(w)|^2)\frac{1-\la z,w\ra}{1-\la w,z\ra}\la\partial_zf,\overline{z-w}\ra\la\bpartial_wg,z-w\ra\\
			&\qquad \qquad \qquad \qquad \qquad \qquad \BKt_\xi(w)\BKt_w(z)\BKt_z(\xi)\intd\lambda_t(w)\intd\lambda_t(z)\bigg\}\intd\lambda_t(\xi).
		\end{flalign*}
		It follows from our assumption that $f,g$ satisfy condition 1 and Lemma \ref{lem: Rudin Forelli generalizations} that the integral converges absolutely. Applying Fubini's theorem, we continue the above computation, 
		\begin{flalign*}
			&-\int_{\bn}\int_{\bn}\Phi^{(t)}_{n,1}(|\varphi_z(w)|^2)\frac{1-\la z,w\ra}{1-\la w,z\ra}\la\partial_zf,\overline{z-w}\ra\la\bpartial_wg,z-w\ra\\
			&\qquad \qquad\qquad \qquad\qquad \qquad \bigg\{\int_{\bn}\BKt_\xi(w)\BKt_z(\xi)\intd\lambda_t(\xi)\bigg\}\BKt_w(z)\intd\lambda_t(w)\intd\lambda_t(z)\\
			=&-\int_{\bn}\int_{\bn}\Phi^{(t)}_{n,1}(|\varphi_z(w)|^2)\frac{1-\la z,w\ra}{1-\la w,z\ra}\la\partial_zf,\overline{z-w}\ra\la\bpartial_wg,z-w\ra|\BKt_w(z)|^2\intd\lambda_t(w)\intd\lambda_t(z)\\
			=&\overline{-\int_{\bn}\int_{\bn}\Phi^{(t)}_{n,1}(|\varphi_z(w)|^2)\frac{(1-\la w,z\ra)^2\sum_{i,j=1}^nI^{e_i,e_j}(w-z)\partial_i\bar{g}(w)\bpartial_j\bar{f}(z)}{|1-\la z,w\ra|^2}|\BKt_w(z)|^2\intd\lambda_t(w)\intd\lambda_t(z)}.
		\end{flalign*}
		Applying Lemma \ref{lem: formula F induction} with $\Gamma=\{(e_i,e_j): i,j=1,\ldots,n\}$, $k=1$, we obtain
		\[
		F_{e_i,e_j}(z,w)=(1-\la w,z\ra)^2\partial_i\bar{g}(w)\bpartial_j\bar{f}(z)\BKt_z(w)
		\]
		and also using Lemma \ref{lem: d f g= f g log = f g R} we get
		\begin{flalign*}
			&\int_{\bn}\int_{\bn}\Phi^{(t)}_{n,1}(|\varphi_z(w)|^2)\frac{(1-\la w,z\ra)^2\sum_{i,j=1}^nI^{e_i,e_j}(w-z)\partial_i\bar{g}(w)\bpartial_j\bar{f}(z)}{|1-\la z,w\ra|^2}|\BKt_w(z)|^2\intd\lambda_t(w)\intd\lambda_t(z)\\
			=&\int_{\bn}\int_{\bn}\Phi^{(t)}_{n,1}(|\varphi_z(w)|^2)\frac{\sum_{i,j=1}^nF_{e_i,e_j}(z,w)}{|1-\la z,w\ra|^2}\BKt_w(z)\intd\lambda_t(w)\intd\lambda_t(z)\\
			=&\frac{\BFt_{n+1}\Phi^{(t)}_{n,1}(0)}{B(n,t+1)}\int_{\bn}(1-|z|^2)^{-2}\sum_{i,j=1}^nd_{e_i,e_j}(z)F_{e_i,e_j}(z,z)\intd\lambda_t(z)\\
			&-\int_{\bn\times\bn}\Phi^{(t)}_{n,2}(|\varphi_z(w)|^2)\frac{\sum_{i,j,k,l=1}^nI^{e_i+e_k,e_j+e_l}(z-w)(1-\la z,w\ra)^2\partial_{z_k}\bpartial_{w_l}F_{e_i,e_j}(z,w)}{|1-\la z,w\ra|^4}\\
			&\qquad \qquad \qquad \qquad \qquad \qquad \qquad \qquad \BKt_w(z)\intd\lambda_t(z)\intd\lambda_t(w)\\
			=&\frac{\BFt_{n+1}\Phi^{(t)}_{n,1}(0)}{B(n,t+1)}\int_{\bn}\frac{\sum_{i,j=1}^nd_{e_i,e_j}(z)\partial_i\bar{g}(z)\bpartial_j\bar{f}(z)}{(1-|z|^2)^{n+1+t}}\intd\lambda_t(z)\\
			&-\int_{\bn\times\bn}\Phi^{(t)}_{n,2}(|\varphi_z(w)|^2)\sum_{i,j,k,l=1}^nI^{e_i+e_k,e_j+e_l}(z-w)\bpartial_l\partial_i\bar{g}(w)\partial_k\bpartial_j\bar{f}(z)|\BKt_w(z)|^2\intd\lambda_t(z)\intd\lambda_t(w)\\
			=&\frac{\BFt_{n+1}\Phi^{(t)}_{n,1}(0)}{B(n,t+1)}\cdot\frac{(n-1)!}{B(n,t+1)\pi^n}\cdot\frac{-1}{(2i)^n n!}\int_{\bn}\partial\bar{g}\wedge\bpartial \bar{f}\wedge\bigg[\partial\bpartial\log(1-|w|^2)\bigg]^{n-1}\\
			&-\int_{\bn\times\bn}\Phi^{(t)}_{n,2}(|\varphi_z(w)|^2)L_z\bar{f}(z-w)L_w\bar{g}(z-w)|\BKt_w(z)|^2\intd\lambda_t(z)\intd\lambda_t(w).
		\end{flalign*}
		Therefore, we continue the computation of the semi-commutator using the above calculation,
		\begin{flalign*}
			&\Tr\bigg(\BTt_f\BTt_g-\BTt_{fg}\bigg)\\
			=&\overline{-\int_{\bn}\int_{\bn}\Phi^{(t)}_{n,1}(|\varphi_z(w)|^2)\frac{(1-\la w,z\ra)^2\sum_{i,j=1}^nI^{e_i,e_j}(w-z)\partial_i\bar{g}(w)\bpartial_j\bar{f}(z)}{|1-\la z,w\ra|^2}|\BKt_w(z)|^2\intd\lambda_t(w)\intd\lambda_t(z)}\\
			=&-\frac{\BFt_{n+1}\Phi^{(t)}_{n,1}(0)}{B(n,t+1)}\cdot\frac{(n-1)!}{B(n,t+1)\pi^n}\cdot\frac{-1}{(-2i)^n n!}\int_{\bn}\bpartial g\wedge\partial f\wedge\bigg[-\partial\bpartial\log(1-|w|^2)\bigg]^{n-1}\\
			&+\int_{\bn\times\bn}\Phi^{(t)}_{n,2}(|\varphi_z(w)|^2)L_z f(z-w)L_wg(z-w)|\BKt_w(z)|^2\intd\lambda_t(z)\intd\lambda_t(w)\\
			=&a_{n,t}\int_{\bn}\partial f\wedge\bpartial g\wedge\bigg[\partial\bpartial\log(1-|w|^2)\bigg]^{n-1}\\
			&+\int_{\bn\times\bn}\rho_{n,t}(|\varphi_z(w)|^2)L_z f(z-w)L_wg(z-w)\frac{\intd m(z,w)}{|1-\la z,w\ra|^{2n+2}}, 
		\end{flalign*}
		where
		\[
		a_{n,t}=\frac{\BFt_{n+1}\Phi^{(t)}_{n,1}(0)}{\big(B(n,t+1)^2\big)n(2\pi i)^n}\quad\text{and}\quad\rho_{n,t}(s)=\bigg(\frac{(n-1)!}{\pi^nB(n,t+1)}\bigg)^2(1-s)^t\Phi^{(t)}_{n,2}(s). 
		\]
		By Lemma \ref{lem: FPhi(0) expression}, we have the following estimate, 
		\begin{equation}\label{eqn: proof ant}
			a_{n,t}=\frac{-\int_0^1(1-s)^{n-1}s^t\ln s\intd s}{\big(B(n,t+1)^2\big)n(2\pi i)^n}=\frac{n!t^{-n-1}+o(t^{-n-2})}{\big(B(n,t+1)^2\big)n(2\pi i)^n}=\frac{t^{n-1}}{(n-1)!(2\pi i)^n}+o(t^{n-2}).
		\end{equation}
		By Lemma \ref{lem: Phi n 2}, we have the following formula,
		\begin{equation}\label{eqn: proof rhont}
			\rho_{n,t}(s)=s^{-n-1}\sum_{k=1}^n\frac{(n-1)!\Gamma^2(n+t+1)}{(n-k)!\Gamma(t+1+k)\Gamma(t+1)\pi^{2n}}\int_s^1F(s,x)x^{n-k-1}(1-x)^{t+k-1}\intd x.
		\end{equation}
		This proves Equation \eqref{eqn: trace formula high dim}. It remains to prove \eqref{eqn: asymp trace formula high dim}. By our assumption on Condition 2, Lemmas \ref{lem: Mobius basics},  \ref{lem: Rudin Forelli generalizations} and \ref{lem: F(s,x) integral estimate}, we have the following estimates, 
		\begin{flalign*}
			&t^{1-n}\bigg|\int_{\bn\times\bn}\rho_{n,t}(|\varphi_z(w)|^2)L_z f(z-w)L_wg(z-w)\frac{\intd m(z,w)}{|1-\la z,w\ra|^{2n+2}}\bigg|\\
			\lesssim&t^{1-n}\int_{\bn}\int_{\bn}\frac{\rho_{n,t}(|\varphi_z(w)|^2)|\varphi_z(w)|^4}{|1-\la z,w\ra|^{n+2-\epsilon}}\intd m(w)\intd m(z)\\
			\xlongequal{\zeta=\varphi_z(w)}&t^{1-n}\int_{\bn}\int_{\bn}\rho_{n,t}(|\zeta|^2)|\zeta|^4\frac{|1-\la z,\zeta\ra|^{n+2-\epsilon}}{(1-|z|^2)^{n+2-\epsilon}}\cdot\frac{(1-|z|^2)^{n+1}}{|1-\la z,\zeta\ra|^{2n+2}}\intd m(\zeta)\intd m(z)\\
			=&t^{1-n}\int_{\bn}\rho_{n,t}(|\zeta|^2)|\zeta|^4\int_{\bn}\frac{(1-|z|^2)^{-1+\epsilon}}{|1-\la z,\zeta\ra|^{n+\epsilon}}\intd m(z)\intd m(\zeta)\\
			\lesssim&t^{1-n}\int_{\bn}\rho_{n,t}(|\zeta|^2)|\zeta|^4\ln\frac{1}{1-|\zeta|^2}\intd m(\zeta)\\
			\approx&t^{1-n}\int_0^1s^{n+1}\rho_{n,t}(s)\ln\frac{1}{1-s}\intd s\\
			\lesssim&t^{1-n}\int_0^1s^{n+1}\rho_{n,t}(s)(1-s)^{-1/2}\intd s\\
			=&t^{1-n}\int_0^1s^{n+1}\bigg\{s^{-n-1}\sum_{k=1}^n\frac{(n-1)!\Gamma^2(n+t+1)}{(n-k)!\Gamma(t+1+k)\Gamma(t+1)\pi^{2n}}\\
			&\qquad \qquad \qquad \qquad \qquad \qquad \int_s^1F(s,x)x^{n-k-1}(1-x)^{t+k-1}\intd x\bigg\}(1-s)^{-1/2}\intd s\\
			=&t^{1-n}\sum_{k=1}^n\frac{(n-1)!\Gamma^2(n+t+1)}{(n-k)!\Gamma(t+1+k)\Gamma(t+1)\pi^{2n}}\\
			&\qquad \qquad \qquad \qquad \qquad \qquad \int_0^1\int_s^1F(s,x)x^{n-k-1}(1-x)^{t+k-1}\intd x(1-s)^{-1/2}\intd s\\
			=&t^{1-n}\sum_{k=1}^n\frac{(n-1)!\Gamma^2(n+t+1)}{(n-k)!\Gamma(t+1+k)\Gamma(t+1)\pi^{2n}}\\
			&\qquad \qquad \qquad \qquad \qquad \qquad \int_0^1\bigg\{\int_0^xF(s,x)(1-s)^{-1/2}\intd s\bigg\}x^{n-k-1}(1-x)^{t+k-1}\intd x\\
			\lesssim&t^{1-n}\sum_{k=1}^n\frac{(n-1)!\Gamma^2(n+t+1)}{(n-k)!\Gamma(t+1+k)\Gamma(t+1)\pi^{2n}}\int_0^1x^{n-k+1}(1-x)^{t+k-1}\intd x\\
			=&t^{1-n}\sum_{k=1}^n\frac{(n-1)!\Gamma^2(n+t+1)}{(n-k)!\Gamma(t+1+k)\Gamma(t+1)\pi^{2n}}B(n-k+2,t+k)\\
			=&o(t^{-1})\\
			&\to 0,
		\end{flalign*}
		as $t$ tends to infinity. Combining the above, \eqref{eqn: proof ant}, and \eqref{eqn: trace formula high dim} we obtain \eqref{eqn: asymp trace formula high dim}. This completes the proof.
	\end{proof}

	\section{Applications and Examples}\label{sec: applications and examples}
	We start this section with some applications of Theorem \ref{thm: semi-commutator trace in dim 1}.
	Since $|\varphi_z(w)|=|\varphi_w(z)|$, it follows immediately that the second term in \eqref{eqn: semi-commutator trace in dim 1} is symmetric in the symbols $f$ and $g$. 
	As a consequence, the following trace formula for commutators of Toeplitz operators holds.

	\begin{cor}\label{cor: commutator trace in dim 1}
		Suppose $t>-1$ and $f, g\in\mathscr{C}^2(\overline{\dd})$. Then
		\begin{equation}\label{eqn: commutator trace in dim 1}
			\Tr[\BTt_f, \BTt_g]=\frac{1}{2\pi i}\int_{\dd}\intd f\wedge\intd g.
		\end{equation}
	\end{cor}
	For the case when $t=0$, this result is well-known (cf. \cite{HH1, Zhu2001trace}).
	
	We can apply Theorem \ref{thm: semi-commutator trace in dim 1} to study Hankel operators. Recall that the Hankel operator with symbol $g$ is defined on $L^2(\lambda_t)$ by
	\[
	\BHt_g=(I-\BPt)M_g\BPt,
	\]
	where $\BPt$ is the Bergman projection. By the identity
	\[
	\BTt_f\BTt_g-\BTt_{fg}=-H^{(t)*}_{\bar{f}}H^{(t)}_g,
	\]
	we have
	\[
	\Tr(\BTt_f\BTt_g-\BTt_{fg})=-\Tr(H^{(t)*}_{\bar{f}}\BHt_g)=-\la\BHt_g, \BHt_{\bar{f}}\ra_{\mathcal{S}^2}.
	\]
	Thus \eqref{eqn: semi-commutator trace in dim 1} leads to a formula for the inner product of Hankel operators in the Hilbert-Schmidt class. In particular, it leads to a formula for the Hilbert-Schmidt norm of Hankel operators.
	
	\begin{cor}\label{cor: hankel HS norm in dim 1}
		Suppose $t>-1$ and $g\in\mathscr{C}^2(\overline{\dd})$. Then
		\[
		\|H^{(t)}_g\|_{\mathcal{S}^2}^2=\frac{1}{\pi}\int_{\dd}|\bpartial g|^2\intd m-\int_{\dd^2}\varrho_t(|\varphi_z(w)|^2)\Delta \bar{g}(z)\Delta g(w)\intd m(z,w).
		\]
		where $\varrho_t$ is defined as in Theorem \ref{thm: semi-commutator trace in dim 1}. In particular,
		\[
		\lim_{t\to\infty}\|H^{(t)}_g\|_{\mathcal{S}^2}^2=\frac{1}{\pi}\int_{\dd}|\bpartial g|^2\intd m.
		\]
	\end{cor}
	
	For Hankel operators with real subharmonic symbols, the second term is non-negative. So the following holds.
	
	\begin{cor}\label{cor: Hankel HS norm inequality in dim 1}
		Suppose $t>-1$ and $g\in\mathscr{C}^2(\overline{\dd})$ is real-valued and subharmonic in $\dd$. Then
		\begin{equation}\label{eqn: Hankel HS norm in dim 1}
			\|\BHt_g\|_{\mathcal{S}^2}^2\leq\frac{1}{\pi}\int_{\dd}|\bpartial g|^2\intd m,
		\end{equation}
		with equality holds if and only if $g$ is harmonic in $\dd$.
	\end{cor}

	As explained in the introduction, in this paper we focus more on the trace formula \eqref{eqn: trace formula high dim} and asymptotic trace formula \eqref{eqn: asymp trace formula high dim} of semi-commutators with relatively nice symbols. Nonetheless, the following lemma and the examples that follow show that Condition 1 is a natural condition to work with.

	\begin{lem}\label{lem: gradient product}
		Suppose $n\geq2$ and $f, g$ satisfy Condition 1. Then there exists a constant $C>0$ such that for any $z\in\bn\backslash\{0\}$,
		\begin{equation}\label{eqn: gradient product}
			\big|\la\partial_zf,\bar{\zeta}\ra\la\bpartial_zg,\zeta\ra\big|\leq C\bigg(|P_z(\zeta)|^2+(1-|z|^2)|Q_z(\zeta)|^2\bigg)(1-|z|^2)^{n-2+\epsilon},\quad\forall\zeta\in\cn.
		\end{equation}
		\begin{enumerate}
			\item[(1)] In the special case when $f=\bar{g}$, \eqref{eqn: gradient product} becomes
			\begin{equation}\label{eqn: bar gradient inequality}
				\big|\la\bpartial_zg,\zeta\ra\big|\leq C_1\bigg(|P_z(\zeta)|+(1-|z|^2)^{1/2}|Q_z(\zeta)|\bigg)(1-|z|^2)^{\frac{n-2+\epsilon}{2}},
			\end{equation}
			which is equivalent to Condition 1. Here $C_1$ is another constant.
			\item[(2)] If there are $a, b\geq0$, $a+b>n-2$ such that
			\begin{equation}
				\big|\la\partial_zf,\bar{\zeta}\ra\big|\leq C_2\bigg(|P_z(\zeta)|+(1-|z|^2)^{1/2}|Q_z(\zeta)|\bigg)(1-|z|^2)^a,
			\end{equation}
			\begin{equation}
				\big|\la\bpartial_zg,\zeta\ra\big|\leq C_2\bigg(|P_z(\zeta)|+(1-|z|^2)^{1/2}|Q_z(\zeta)|\bigg)(1-|z|^2)^b
			\end{equation}
			with some constant $C_2$, then $f, g$ satisfy Condition 1.
		\end{enumerate}
		
	\end{lem}
	
	\begin{proof}
		Note that
		\begin{equation}\label{eqn: estimates}
			|1-\la z,w\ra|\approx(1-|z|^2)+(1-|w|^2)+|z-w|^2+|\mathrm{Im}\la z,w\ra|,\ |\varphi_z(w)|^2=\frac{|z-P_z(w)|^2+(1-|z|^2)|Q_z(w)|^2}{|1-\la z,w\ra|^2}.
		\end{equation}

		Take $w=z+\lambda\zeta$ where $\lambda\in\CC$ is sufficiently small. Then by definition,
		\[
		|\varphi_z(w)|^2=|\lambda|^2\frac{|P_z(\zeta)|^2+(1-|z|^2)|Q_z(\zeta)|^2}{|1-\la z,w\ra|^2}.
		\]
		Condition 1 implies
		\[
		\big|\la\partial_zf,\bar{\zeta}\ra\la\bpartial_wg,\zeta\ra\big||\lambda|^2\leq C|\lambda|^2\frac{|P_z(\zeta)|^2+(1-|z|^2)|Q_z(\zeta)|^2}{|1-\la z,w\ra|^2}|1-\la z,w\ra|^{n+\epsilon}.
		\]
		Canceling out $|\lambda|^2$ and letting $\lambda\to0$ we obtain the first inequality in \eqref{eqn: gradient product}. The second inequality is proved similarly. 
		
		On the other hand, suppose $f=\bar{g}$ and \eqref{eqn: bar gradient inequality} holds. Then
		\begin{flalign*}
			\big|\la\bpartial_zg,z-w\ra\big|^2\lesssim&\bigg(|z-P_z(w)|^2+(1-|z|^2)|Q_z(w)|^2\bigg)(1-|z|^2)^{n-2+\epsilon}\\
			\lesssim&\bigg(|z-P_z(w)|^2+(1-|z|^2)|Q_z(w)|^2\bigg)|1-\la z,w\ra|^{n-2+\epsilon}\\
			=&|\varphi_z(w)|^2|1-\la z,w\ra|^{n+\epsilon}.
		\end{flalign*}
		Equivalently,
		\[
		\big|\la\bpartial_wg,z-w\ra\big|^2\lesssim|\varphi_z(w)|^2|1-\la z,w\ra|^{n+\epsilon}
		\]
		and also
		\[
		\big|\la\partial_z\bar{g},\overline{z-w}\ra\big|^2\lesssim|\varphi_z(w)|^2|1-\la z,w\ra|^{n+\epsilon}.
		\]
		Multiplying the two inequalities and taking square root gives Condition 1 for $f=\bar{g}$. This proves (1). Statement (2) is proved in the same way as (1). We omit the details.
	\end{proof}
	
	Similarly, one may give sufficient conditions for Condition 2 in terms of growth rates of second order derivatives. Taking the case $f=\bar{g}$ for example, we have the following.
	
	\begin{lem}
Suppose $n\geq2$ and $g\in\mathscr{C}^2(\bn)$. If $g$ satisfies \eqref{eqn: bar gradient inequality} and for some constant $C>0$ and $a\geq\max\{0, \frac{n}{2}-2\}$,
\begin{equation*}
|L_zg(\zeta)|\leq C\bigg(|P_z(\zeta)|^2+(1-|z|^2)|Q_z(\zeta)|^2\bigg)(1-|z|^2)^a,
\end{equation*}
then $f=\bar{g}$ and $g$ satisfy Condition 1 and 2.
	\end{lem}
	
The Schatten class criterion of Hankel operators is throughly studied, c.f. \cite{Ar-Fi-Pe:Hankel, Ar-Fi-Ja-Pe:hankel, Li-Lu:HankelinSp, Zheng:Hankel, Zhu:SchattenHankel}. There are also some results on the Schatten norms of Hankel operator with anti-holomorphic symbols, c.f. \cite{Ja-Up-Wa:SchattenNormIdentity}\cite{Xia:SchattenNormIdentity}. In \cite[Theorem 3.1]{Li-Lu:HankelinSp}, Li and Luecking gave a criterion for Hankel operators to be in $\mathcal{S}^p$. Our condition in (1) of Lemma \ref{lem: gradient product} is consistent with that of $f_2$ in Li and Luecking's Theorem 3.1 when $p=2$. One can also check using \cite[Theorem 3.1]{Li-Lu:HankelinSp} that when (2) of Lemma \ref{lem: gradient product} holds, $\BHt_{\bar{f}}\in\mathcal{S}^p$ and $\BHt_g\in\mathcal{S}^q$ for some $\frac{1}{p}+\frac{1}{q}=1$. So the trace class membership of $\BTt_f\BTt_g-\BTt_{fg}$ follows from the identity $\BTt_f\BTt_g-\BTt_{fg}=-H^{(t)*}_{\bar{f}}\BHt_g$. The converse, however, is not true: there are symbols $f, g$ such that $\BHt_{\bar{f}}$ and $\BHt_g$ belong only to bigger Schatten classes but their product belongs to the trace class. The following lemma gives us a clue.
	
	\begin{lem}\label{lem: f g compact support}
		Suppose $f, g\in\mathscr{C}^2(\bn)$ are bounded and have bounded first and second order derivatives. If $\supp fg$ is a compact subset in $\bn$ then $f, g$ satisfy Conditions 1 and 2 in Theorem \ref{thm: higher dimensions}. In particular \eqref{eqn: trace formula high dim} and \eqref{eqn: asymp trace formula high dim} hold.
	\end{lem}
	
	\begin{proof}
		By \eqref{eqn: estimates}, in this case $|1-\la z,w\ra|$ is bounded away from $0$ for $(z,w)\in\supp f\times\supp g$. From this it is easy to verify Conditions 1 and 2.
	\end{proof}

	In the case when $\supp f$ and $\supp g$ do meet on the boundary, Condition 1 gives us an idea of how much decay is needed when they meet. See the following example.
	
	\begin{exam}\label{example}
		Let $\epsilon>0$ and $\psi$ be a $\mathscr{C}^1$ function on $\mathbb{R}$ such that 
		\[
		\psi'(s)=0\text{ for }s<0,\text{ and }|\psi'(s)|\lesssim s^{1+\epsilon}\text{ for }s\geq0.
		\]
		Let $n=2$ and
		\[
		f(z)=\psi(|z_1|^2-|z_2|^2),\quad g(z)=\psi(|z_2|^2-|z_1|^2).
		\]
		Then we compute
		\[
		|\partial_zf|\begin{cases}
			=0,&\text{if }|z_1|<|z_2|\\
			\lesssim(|z_1|^2-|z_2|^2)^{1+\epsilon},&\text{if }|z_1|\geq|z_2|
		\end{cases},
		\quad
		|\bpartial_wg|\begin{cases}
			=0,&\text{if }|w_2|<|w_1|\\
			\lesssim(|w_2|^2-|w_1|^2)^{1+\epsilon},&\text{if }|w_2|\geq|w_1|
		\end{cases}
		\]
		Whenever $|\partial_zf||\bpartial_wg|$ is non-zero we have $|z_1|>|z_2|$ and $|w_2|>|w_1|$, in which case
		\[
		|z-w|\approx|z_1-w_1|+|z_2-w_2|\geq|z_1|-|w_1|+|w_2|-|z_2|=(|z_1|-|z_2|)+(|w_2|-|w_1|).
		\]
		So we have the following estimate
		\[
		|\partial_zf||\bpartial_wg|\lesssim|z-w|^{2+2\epsilon}.
		\]
		Using the above inequality, we reach the following estimate, 
		\[
		|\la\partial_zf,\overline{z-w}\ra||\la\bpartial_wg,z-w\ra|\lesssim|z-w|^2|\partial_zf||\bpartial_wg|\lesssim|z-w|^{4+2\epsilon}\lesssim|\varphi_z(w)|^2|1-\la z,w\ra|^{2+\epsilon}.
		\]
		So Condition 1 is satisfied and by Theorem \ref{thm: higher dimensions}, $\BTt_f\BTt_g-\BTt_{fg}$ is in the trace class.
	\end{exam}

	\section{Appendix I: A Formula of Bochner-Martinelli Type}\label{appendix I: a formula of BM type}
	Recall that $R=\sum_{i=1}^nz_i\partial_{z_i}$ is the radial derivative operator, and $\bar{R}=\sum_{i=1}^n\bar{z}_i\bpartial_{z_i}$. In this appendix we prove the following lemma.
	
	\begin{lem}\label{lem: formula rSn 0}
		Suppose $r>0$, $\alpha, \beta\in\ind$, $|\alpha|\geq|\beta|$ and $v\in\mathscr{C}^1(\overline{r\bn})$. Then
		\begin{flalign}\label{eqn: formula rSn 0}
			&\int_{r\sn}z^\alpha\bar{z}^\beta v(z)\intd\sigma_r(z)\\
			=&a_{\alpha,\beta}\sigma_{2n-1}r^{2|\beta|+2n-1}v(0)+2r^{2|\beta|+2n-1}\int_{r\bn}\frac{z^\alpha\bar{z}^\beta}{|z|^{2|\beta|+2n}}\bar{R}v(z)\intd m(z).\nonumber
		\end{flalign}
		Here
		\[
		a_{\alpha,\beta}=\delta_{\alpha,\beta}\frac{(n-1)!\alpha!}{(n-1+|\alpha|)!}.
		\]
	\end{lem}
	Lemma \ref{lem: formula rSn 0} can be verified directly on $v(z)=z^\gamma\bar{z}^\iota$, and then using approximation of $v$ by polynomials. For future reference, we show in the rest of this appendix that it can be viewed as a special case of a Bochner-Martinelli type formula (see Proposition \ref{prop: formula BM generalized} below).
	
	For a $(p, q)$-form $u=\sum_{|I|=p, |J|=q}u_{I,J}\intd z_I\wedge\intd z_J$,
	\[
	\partial u=\sum_{k=1}^n\sum_{|I|=p, |J|=q}\partial_ku_{I,J}\intd z_k\wedge\intd z_I\wedge\intd z_J,\quad \bpartial u=\sum_{k=1}^n\sum_{|I|=p, |J|=q}\bpartial_ku_{I,J}\intd \bar{z}_k\wedge\intd z_I\wedge\intd z_J.
	\]
	Then $\intd=\partial+\bpartial$ is the exterior derivative.
	
	~
	
	In some of the estimates, we may abuse notations and use $\partial f$, $\bpartial f$ to denote holomorphic, and anti-holomorphic gradient of a $\mathscr{C}^1$ function $f$, i.e.,
	\[
	\partial f(z)=(\partial_1 f(z),\partial_2 f(z),\ldots, \partial_n f(z)),\quad\bpartial f(z)=(\bpartial_1 f(z), \bpartial_2 f(z),\ldots, \bpartial_n f(z)),
	\]
	considered as column vectors.
	
	~
	
	\noindent{\bf Cauchy Formula.} Let $\Omega\subset\CC$ be a bounded open set with $\mathscr{C}^1$ boundary and $0\in\Omega$. Then for every $v\in\mathscr{C}^1(\overline{\Omega})$,
	\[
	\frac{1}{2\pi i}\int_{\partial\Omega}\frac{v(z)}{z}\intd z=v(0)-\frac{1}{2\pi i}\int_{\Omega}\frac{\bpartial v(z)}{z}\intd z\wedge\intd\bz.
	\]
	
	~
	
	Its generalization to higher dimensions is the Bochner-Martinelli Formula.\\

	\noindent{\bf Bochner-Martinelli Formula.} Let $\Omega\subset\cn$ be a bounded open set with $\mathscr{C}^1$ boundary and $0\in\Omega$. Then for every $v\in\mathscr{C}^1(\overline{\bn})$,
	\begin{equation}\label{eqn: formula BM}
		\int_{\partial\Omega}v(z)k_{BM}(z)=v(0)+\int_{\Omega}\bpartial v(z)\wedge k_{BM}(z).
	\end{equation}
	Here $k_{BM}$ is the Bochner-Martinelli kernel, defined by
	\begin{equation}\label{eqn: kBM defn}
		k_{BM}(z)=\frac{(-1)^{n-1}}{(2\pi i)^n}|z|^{-2n}\partial|z|^2\wedge\big(\partial\bpartial|z|^2\big)^{n-1}.
	\end{equation}
	In this section, we prove a generalization of the above Bochner-Martinelli Formula (\ref{eqn: formula BM}). First, we review the definition of currents on $\cn$.
	\begin{defn}
		For $p, q=0,\ldots,n$, denote $\mathscr{D}^{p,q}$ the locally convex space of smooth $(p,q)$-forms on $\cn$ with compact support. The topology of $\mathscr{D}^{p,q}$ is defined by the collection of semi-norms
		\[
		\|u\|_{\mathscr{D}^{p,q}(K),N}:=\sum_{|I|=p,|J|=q}\sup_{|\alpha|+|\beta|\leq N}\sup_{z\in K}|\partial^\alpha\bpartial^\beta u_{I,J}(z)|,\quad u=\sum_{|I|=p,|J|=q}u_{I,J}(z)\intd z_I\wedge\intd\bz_J,
		\]
		where $N$ ranges over all positive integers, and $K$ is any compact subset in $\cn$.
		The space of currents of bidegree $(p,q)$, denoted by ${\mathscr{D}'}^{p,q}$, is the dual space of $\mathscr{D}^{n-p,n-q}$, endowed with the weak* topology. The currents in ${\mathscr{D}'}^{p,q}$ can be viewed as $(p,q)$-forms with distribution coefficients. In particular, any $(p,q)$-form with locally integrable coefficients is a current of bidegree $(p,q)$. With the identification of Lebesgue measure $\intd m(z)$ with the Euclidean volume form
		\[
		\intd v:=\frac{1}{(-2i)^n}\intd z_1\wedge\intd\bz_1\wedge\ldots\wedge\intd z_n\wedge\intd\bz_n,
		\]
		a distribution $T$ on $\cn$ can be viewed as either a current of bidegree $(0,0)$ or $(n,n)$: for $h$ on $\cn$ smooth and compactly supported,
		\[
		\la T, h\ra=\la T, h\intd v\ra.
		\]
		Differential operators acts on currents by duality and are continuous with respect to the weak* topology.
	\end{defn}
	
	\begin{defn}
		For multi-indices $\alpha, \beta\in\ind$, define
		\begin{equation}\label{eqn: H alpha beta defn}
			H_{\alpha,\beta}(z)=\lim_{\epsilon\to0^+}\bigg(\frac{z^\alpha\bz^\beta}{|z|^{2|\beta|}}k_{BM}(z)\bigg)\bigg|_{\cn\backslash\epsilon\bn},
		\end{equation}
		in the sense of current.
		
		In the case when $|\beta|\leq|\alpha|$, $H_{\alpha,\beta}$ has locally integrable distributions. So we can simply write
		\[
		H_{\alpha,\beta}(z)=\frac{z^\alpha\bz^\beta}{|z|^{2|\beta|}}k_{BM}(z).
		\]
		Only this case will be used in this paper. For completeness and future reference, we include the case when $|\beta|>|\alpha|$. In this case, the current $H_{\alpha,\beta}$ has coefficient distributions which are not locally integrable, and we need to define it in the style of a principal value. For any $1\leq j\leq n$ and $\mathscr{C}^\infty$, compactly supported function $h$ on $\cn$, take the Taylor expansion
		\[
		h(z)=\sum_{|\gamma_1|+|\gamma_2|\leq |\beta|-|\alpha|}\frac{\partial^{\gamma_1}\bpartial^{\gamma_2}h(0)}{\gamma_1!\gamma_2!}z^{\gamma_1}\bz^{\gamma_2}+O(|z|^{|\beta|-|\alpha|+1}).
		\]
		For each $\epsilon>0$, the current inside the limit sign of \eqref{eqn: H alpha beta defn} vanishes on every term except for $O(|z|^{|\beta|-|\alpha|+1})$. Thus the current $H_{\alpha,\beta}$ for $|\beta|>|\alpha|$ is well-defined.
	\end{defn}

	The standard Bochner-Martinelli formula follows from Stoke's Theorem and the following identity.
	\begin{equation}\label{eqn: d''kBM}
		\bpartial k_{BM}=\delta_0,
	\end{equation}
	where $\delta_0$ is the point mass at $0$. Standard arguments show that the following holds.
	\begin{lem}\label{lem: d''H_alpha_beta}
		We have
		\begin{equation}\label{eqn: d''H_alpha_beta}
			\bpartial H_{\alpha,\beta}=(-1)^{l-k}a_{\alpha,\beta}\partial^{\beta-\alpha}\delta_0.
		\end{equation}
		Here
		\[
		a_{\alpha,\beta}=
		\begin{cases}
			\frac{(n-1)!\alpha!}{(n-1+|\beta|)!},&\text{if } \alpha\leq\beta,\\
			0,&\text{otherwise}.
		\end{cases}
		\]
	\end{lem}

	\begin{prop}\label{prop: formula BM generalized}
		Let $\Omega\subset\cn$ be a bounded open set with $\mathscr{C}^1$ boundary and $0\in\Omega$. Then for multi-indices $\alpha, \beta\in\ind$ and $v\in\mathscr{C}^1(\overline{\Omega})$,
		\begin{equation}\label{eqn: formula BM generalized}
			\int_{\partial\Omega}v(z)H_{\alpha,\beta}(z)=a_{\alpha,\beta}\partial^{\beta-\alpha}v(0)+\int_{\Omega}\bpartial v(z)\wedge H_{\alpha,\beta}(z).
		\end{equation}
	\end{prop}
	
	\begin{proof}[{\bf Proof of Proposition \ref{prop: formula BM generalized}}]
		By Lemma \ref{lem: d''H_alpha_beta} and Stokes' Theorem for currents, we have the following equations
		\begin{flalign*}
			&\int_{\partial\Omega} v(z)H_{\alpha,\beta}(z)=\int_{\Omega}\intd\bigg(v(z)H_{\alpha,\beta}(z)\bigg)=\int_{\Omega}\bpartial\bigg(v(z)H_{\alpha,\beta}(z)\bigg)\\
			=&\int_{\Omega}v(z)\bpartial H_{\alpha,\beta}(z)+\int_{\Omega}\bpartial v(z)\wedge H_{\alpha,\beta}(z)=a_{\alpha,\beta}\partial^{\beta-\alpha}v(0)+\int_{\Omega}\bpartial v(z)\wedge H_{\alpha,\beta}(z).
		\end{flalign*}
		This completes the proof of Proposition \ref{prop: formula BM generalized}.
	\end{proof}
	Taking $\Omega=r\bn$ and assuming $|\alpha|\geq|\beta|$ in Proposition \ref{prop: formula BM generalized} gives Lemma \ref{lem: formula rSn 0}.

	\section{Appendix II: Auxiliary Functions and Operations}\label{appendix II: auxiliary functions and operations}
	In Section \ref{sec: trace formulas on the disk}, the integral operations $\BFt, \BGt$ simplify our computation. To work in higher dimensions, it is necessary to extend those integral operations and establish some basic properties. This is the goal of the current section.
	
	\begin{defn}\label{defn: BFt BGt}
		For $t\in\mathbb{R}$, denote
		\[
		\phi_t(s)=(1-s)^t.
		\]
		Suppose $\phi:(0,1)\to[0,\infty)$ is a measurable function. For a positive integer $m$ and any $t>-1$, define the operations on $\phi$
		\begin{equation}
			\BFt_m\phi(s)=\int_s^1 r^{m-1}\phi(r)(1-r)^t\intd r\in[0,\infty],
		\end{equation}
		and
		\begin{equation}
			\BGt_m\phi(s)=\frac{1}{s^m\phi_{t+1}(s)}\BFt_m\phi(s)=\frac{\int_s^1 r^{m-1}\phi(r)(1-r)^t\intd r}{s^m(1-s)^{t+1}}\in[0,\infty].
		\end{equation}
		For any $t>-1$, inductively define the functions
		\[
		\Phi_{n,0}^{(t)}\equiv1,\quad\Phi_{n,k+1}^{(t)}=M_{\phi_1}\big(\BGt_{n+k}\big)^2\Phi_{n,k}^{(t)}.
		\]
		Equivalently,
		\begin{equation}
			\Phi_{n,k}^{(t)}=M_{\phi_1}\big(\BGt_{n+k-1}\big)^2\ldots M_{\phi_1}\big(\BGt_n\big)^21.
		\end{equation}
	\end{defn}

	It is straightforward to verify that the following estimates hold.
	\begin{lem}\label{lem: G estimates}
		Suppose $a>m$ is not an integer and $b\geq0$. Suppose $\phi:(0,1)\to[0,\infty)$ is measurable and
		\[
		\phi(s)\lesssim s^{-a}(1-s)^{b}.
		\]
		Then
		\[
		\BGt_m\phi(s)\lesssim s^{-a}(1-s)^{b}.
		\]
	\end{lem}
	
	As a consequence, the following hold.
	\begin{lem}\label{lem: Phi asymptotic behavior}
		For any $t>-1$ and integers $n>0, k\geq0$,
		\begin{equation}\label{eqn: Phi behavior}
			\Phi_{n,k}^{(t)}(s)\lesssim s^{-n-k+\frac{1}{2}}(1-s)^k,
		\end{equation}
		and
		\begin{equation}\label{eqn: GPhi behavior}
			\BGt_{n+k}\Phi^{(t)}_{n,k}(s)\lesssim s^{-n-k-\frac{1}{2}}(1-s)^k.
		\end{equation}
	\end{lem}
	
	\begin{lem}\label{lem: FG properties}
		For any $t>-1$ and positive integers $m, k$ and $\phi:(0,1)\to[0,\infty)$ we have
		\begin{equation}\label{eqn: BFt induction 1}
			\BFt_{m+k}M_{\phi_1}\BGt_m\phi(0)=\frac{1}{k}\BFt_{m+k}\phi(0),	
		\end{equation}
		\begin{equation}\label{eqn: BFt induction 2}
			\BFt_{m+k}\BGt_m\phi(0)=\sum_{j=0}^\infty\frac{1}{k+j}\BFt_{m+k+j}\phi(0),
		\end{equation}
		\begin{equation}\label{eqn: BFt induction 3}
			\BFt_m\phi=\BFt_mM_{\phi_1}\phi(0)+\BFt_{m+1}\phi,
		\end{equation}
		and
		\begin{equation}\label{eqn: BFt 1}
			\BFt_m1(0)=B(m,t+1).
		\end{equation}
	\end{lem}
	
	\begin{proof}
		The proof is a simple application of Fubini's Theorem. By definition, we have the following computation for $\BFt_{m+k}M_{\phi_1}\BGt_m\phi(0)$,
		\begin{flalign*}
			&\BFt_{m+k}M_{\phi_1}\BGt_m\phi(0)\\
			=&\int_0^1r^{m+k-1}(1-r)^{t+1}\BGt_m\phi(r)\intd r\\
			=&\int_0^1r^{k-1}\int_r^1s^{m-1}(1-s)^t\phi(s)\intd s\intd r\\
			=&\int_0^1\int_0^sr^{k-1}\intd rs^{m-1}(1-s)^t\phi(s)\intd s\\
			=&\frac{1}{k}\int_0^1s^{m+k-1}(1-s)^t\phi(s)\intd s\\
			=&\frac{1}{k}\BFt_{m+k}\phi(0).
		\end{flalign*}
		This proves \eqref{eqn: BFt induction 1}. We prove \eqref{eqn: BFt induction 2} by Lemma \ref{lem: Phi asymptotic behavior} using the expansion
		\[
		\frac{1}{1-s}=\sum_{j=0}^\infty s^j.
		\]
		Finally, we arrive at the following equation,
		\[
		\BFt_m1(0)=\int_0^1r^{m-1}(1-r)^t\intd r=B(m,t+1).
		\]
		This completes the proof of Lemma \ref{lem: FG properties}.
	\end{proof}

	\bibliographystyle{plain}
	\bibliography{referencesemi}

\end{document}